\documentclass[11pt]{amsart}
\usepackage[margin=1.2in]{geometry}
\usepackage{amsmath,amsfonts,amssymb,amsthm,mathtools}
\usepackage{cases}
\usepackage{hyperref}
\hypersetup{colorlinks=true,citecolor=blue!70!black,linkcolor=red!80!black,linktocpage=true}
\usepackage[capitalise,nameinlink]{cleveref}
\crefname{enumi}{part}{parts}
\usepackage{enumitem}
\usepackage{url}
\usepackage[all]{xy}
\usepackage{tikz-cd}
\usepackage{array} 
\usepackage{mathtools}
\usepackage{faktor}
\usepackage{caption}
\usepackage{verbatim,eucal}
\usepackage{xfrac}

\crefname{paragraph}{observation}{observations}
\Crefname{paragraph}{Observation}{Observations}
\makeatother 

\numberwithin{equation}{section}

\newtheorem{theorem}[equation]{Theorem} 

\newtheorem{proposition}[equation]{Proposition}
\newtheorem{lemma}[equation]{Lemma} 
\newtheorem{cor}[equation]{Corollary}

\newtheorem{example}[equation]{Example}
\newtheorem{remark}[equation]{Remark}

\newcommand{\ts}{\hspace{.1ex}}
\newcommand{\curlywedgeup}{
\raisebox{.25ex}{\scalebox{.5}{ $\!\!\!\curlywedge$}}}
\newcommand{\bigcurlywedge}{\text{\huge $\curlywedge$}}
\newcommand{\medcurlywedge}{\text{\Large $\curlywedge$}}

\newcommand{\Wedge}{\wedge}
\newcommand{\cross}{\times}

\newcommand{\volv}{\textnormal{vol}_V}
\newcommand{\volx}{\textnormal{vol}_{V^*}}
\newcommand{\coef}{\textnormal{Coef}}
\newcommand{\FF}{\mathbb{F}}
\newcommand{\SL}{\text{SL}}
\newcommand{\GL}{\text{GL}}

\newcommand{\A}{\mathcal A}
\newcommand{\B}{\mathcal B}
\newcommand{\F}{\mathcal F}
\newcommand{\R}{\mathcal R}

\newcommand{\ep}{\varepsilon}
\newcommand{\om}{\omega}
\newcommand{\ot}{\otimes}
\newcommand{\del}{\partial}
\newcommand{\ZZ}{\mathbb{Z}}
\newcommand{\CC}{\mathbb{C}}
\newenvironment{smallpmatrix}
  {\left(\begin{smallmatrix}}
  {\end{smallmatrix}\right)}
\DeclareMathOperator{\characteristic}{char}

\newcommand{\rank}{\text{rank}}
\newcommand{\sbinom}[3][0.8]{\scalebox{#1}{$\binom{#2}{#3}$}}
\newenvironment{smallarray}[1]
 {\null\,\vcenter\bgroup\scriptsize\tiny
  \arraycolsep=.2em
  \hbox\bgroup$\array{@{}#1@{}}}
 {\endarray$\egroup\egroup\,\null}

\begin{document}

\begin{abstract}
Much of the captivating numerology surrounding finite reflection groups 
stems from Solomon's celebrated 1963 theorem
describing invariant differential forms.
Invariant 
differential 
derivations also exhibit
fascinating numerology
over the complex numbers
linked to 
rational Catalan combinatorics.
We explore the analogous theory
over arbitrary fields,
in particular,
when the characteristic
of the underlying field 
divides the order of 
the acting reflection group and
the conclusion of Solomon's Theorem 
may fail.
Using results of Broer and Chuai,
we give a Saito criterion
(Jacobian criterion)
for finding a basis of
differential derivations
invariant under a finite group that distinguishes
certain cases over
fields of characteristic
$2$.
We show that the reflecting hyperplanes lie in a single orbit and demonstrate a duality of exponents and coexponents when the transvection root spaces of a reflection group
are maximal.
A set of basic derivations are used to construct a basis of invariant differential derivations with a twisted wedging in this case.
We obtain explicit bases
for 
the special linear 
groups
$\SL(n,q)$ and general
linear groups $\GL(n,q)$, and all groups in between.
\end{abstract}

\title[Invariant differential derivations for
reflection groups]{Invariant 
differential derivations\\ for 
reflection groups\\
in positive
characteristic}

\date{April 7, 2023}

\author{D.\ Hanson}
\address{Department of Mathematics, University of North Texas,
Denton, Texas 76203, USA}
\email{dillon.hanson@unt.edu}
\author{A.\ V.\ Shepler}
\address{Department of Mathematics, University of North Texas,
Denton, Texas 76203, USA}
\email{ashepler@unt.edu}
\thanks{The second author was partially supported by Simons grants 
429539 and 949953}

\maketitle
\section{Introduction}
Solomon \cite{Solomon} showed that the set of differential forms invariant under the action of a complex reflection group forms a free exterior algebra.
The situation is more subtle
over an arbitrary field, especially
when
the characteristic of the underlying
field $\FF$ divides the order
of the acting group, the
so-called {\em modular setting}.
Zalesskii and Sere\v{z}kin~\cite{ZS}
classified the irreducible
reflection groups over fields of positive
characteristic,
but not every reflection group is the sum of irreducible reflection groups,
and many interesting examples
are reducible with
nondiagonalizable reflections.
Hartmann \cite{Hartmann}
showed that the conclusion
of
Solomon's Theorem 
holds for a group  generated
by diagonalizable reflections
whose ring of invariant polynomials
forms a polynomial algebra.
Hartmann and the second author
\cite{HartmannShepler} 
extended this work 
to exhibit the space of
invariant differential forms
as a free exterior algebra
via a twisted wedge product
when the transvection root spaces are maximal.
Such groups include $\SL_n(\FF_q)$ and $\GL_n(\FF_q)$
for a finite field $\FF_q$, and we explore
these groups as analogs of Coxeter
and 
well-generated complex reflection groups.
We assume all reflection groups are finite.
 
Recently,
attention has turned to
{\em differential derivations}
as their invariants under a reflection group
arise in 
Catalan combinatorics
with connections to
rational
Cherednik algebras, 
symplectic reflection algebras,
and Lie theory 
(e.g., see
\cite{Gordon2003,BerestEtingofGinzburg2003,BessisReiner2011,Rhoades2014,ParkingSpaces,ReinerShepler,DCPP,ReinerSommersShepler}). %
The differential derivations invariant under the action of a well-generated complex reflection group 
constitute a free module
over a certain subalgebra
of the invariant differential 
forms,
and
associated Hilbert series give Kirkman 
numbers
(see \cite{ReinerShepler,ReinerSommersShepler}).

We investigate the
case over 
an arbitrary field $\FF$ here.
We examine  the set 
$(S\otimes \wedge V^*\otimes V)^G$ of
differential derivations 
invariant under a finite
group $G$ acting linearly on a 
finite-dimensional vector space $V=\FF^n$,
with symmetric algebra
$S=S(V^*)$, a polynomial ring.
We 
include the modular setting
when $\characteristic(\FF)$
divides $|G|$.
Broer and Chuai \cite{BroerChuai} 
used ramifications over prime
ideals to give
a general Jacobian criterion.
This criterion requires
a full description of the invariant theory
for 
groups fixing a single hyperplane.
Finding this description
may be trivial when all group elements are diagonalizable but
often is a sticking point when working over arbitrary fields.
Here, we require a rigorous analysis of the actions of transvections on differential derivations (see \cref{appendix}).
We develop a Saito criterion in terms of pointwise stabilizers for determining whether a set of homogeneous elements is a basis:
\begin{theorem}
\label{SaitoThmIntro}
Consider a finite group $G\subset\GL(V)$ acting on $V=\FF^n$.
For a set $\B$ of $n\tbinom{n}{k}$ homogeneous elements in $(S\ot \wedge^k V^*\ot V)^G$, the following are equivalent:

\begin{enumerate}
\setlength{\itemindent}{-2.5ex}
\item[a)]
$(S\otimes \wedge^k V^*\otimes V)^G$ is a free $S^G$-module with basis $\B$.

\item [b)] The coefficient matrix of $\B$ has determinant 
\rule{0ex}{4ex}
$ Q_{\rule[0ex]{0ex}{1ex} 
}^{\sbinom{n-1}{k}} Q_{\det}^{(n-1)\sbinom{n-1}{k-1}} Q_k$ up to a nonzero scalar.
\rule[0ex]{0ex}{0ex}

\item [c)]
$\B$ is independent over $\F(S)$
and 
$\displaystyle
\sum_{\eta\in\B} \deg\eta=\sum_{H\in\A} \tbinom{n-1}{k}+(e_H-1)(n-1)\tbinom{n-1}{k-1}+e_H a_{H,k}\, .
$
\end{enumerate}
\end{theorem}
\noindent
Here, $\F(S)$ is the field of fractions of
$S$,
$e_H$ records the maximal order of a diagonalizable reflection in $G$ about each $H$ in the collection 
$\A$
of reflecting hyperplanes 
of $G$,
the polynomial $Q_k$ in $S$ (see \cref{Q_kdef}) depends on the transvection root space of each $H$,
and the nonnegative integers
$a_{H,k}$ (see \cref{adef})
depend additionally
on the characteristic of $\FF$ in a subtle way.

We argue that reflection groups
 with transvection roots spaces
all maximal, such as
groups $G$
with $\SL_n(\FF_q)
\subset G \subset \GL_n(\FF_q)$ (see \cref{slglsection}),
serve as analogues
of the duality (well-generated)
complex reflection groups
with {\em Coxeter number}
given as the number
of reflecting hyperplanes
times the maximal order of a diagonalizable reflection in
the group (see
\cref{q-Catalan}
and \cref{CoxeterNumber}).
The following result provides the structure of the invariant differential derivations for this class of reflection groups.
\begin{theorem}
\label{AllButOneInto}
Let $G\subset \GL(V)$ be a reflection group with transvection root spaces all maximal 
and $\characteristic \FF\neq 2$.
Suppose $(S\ot V)^G$ is a free $S^G$-module with basic derivations $\theta_1,\ldots, \theta_n$ 
and dual 1-forms $\om_1,\ldots, \om_n$. 
Then $(S\ot \wedge V^*\ot V)^G$ is a free $S^G$-module with basis
$$
\big\{d\theta_E\big\}
\cup
\big\{\omega_I^{\curlywedgeup}\, \theta_1,\dots,\omega_I^{\curlywedgeup}\, \theta_n:
I\subset[n]\big\}
\setminus 
\big\{\om_r \theta_r \big\}
\quad
\text{for any  $r=1,\dots,n$}.
$$
\end{theorem}
\noindent 
We use the exterior derivative of the Euler derivation, $d\theta_E=\sum_{i=1}^n 1\otimes x_i\otimes v_i$,
dual 1-forms
$\omega_1,\dots,\omega_n$
constructed via an operator
related to the
Hodge dual
(see \cref{dualderivations}), 
and twisted wedge products $\omega_I^{\curlywedgeup}$ (see \cref{twistedwedgeproduct}).

\begin{example}
\label{Baby}
{\em 
For the reflection group $G=\langle\begin{smallpmatrix}
1&1\\0&1
\end{smallpmatrix}\rangle$ acting on $V=\FF^2$
with $\characteristic\FF=p>2$,
$$
\begin{aligned}
&\text{basic derivations}
&&\theta_1=1\ot v_1,
&&\hspace{-0.7in}\theta_2 = x_1\ot v_1 + x_2\ot v_2
&&\text{ generate } (S\ot V)^G \text{ and}
\\
&\text{dual $1$-forms} 
&&\omega_1=x_2\ot x_1 - x_1\ot x_2,
&&\omega_2=-1\ot x_2
&&\text{ generate } (S\ot V^*)^G
\end{aligned}
$$
as free $S^G$-modules. Then $(S\ot \wedge V^*\ot V)^G$ is a free $S^G$-module with basis
$$
\begin{aligned}
\theta_1&=1\ot 1\ot v_1, \ \  
\theta_2=x_1\ot 1\ot v_1 + x_2\ot 1\ot v_2,
\ \ 
d\theta_E=1\ot x_1\ot v_1 + 1\ot x_2\ot v_2,
\\
\omega_1\theta_1&=x_2\ot x_1\ot v_1 - x_1\ot x_2\ot v_1,\ \  
\omega_2\theta_1=-1\ot x_2\ot v_1,
\\
\omega_1\theta_2&=x_1x_2\ot x_1\ot v_1 + x_2^2\ot x_1\ot v_2 - x_1^2\ot x_2\ot v_1 - x_1x_2\ot x_2\ot v_2,\\
(\omega_1\curlywedge \omega_2)\theta_1&= -1\ot x_1\wedge x_2\ot v_1, 
\ \ 
(\omega_1\curlywedge\omega_2)\theta_2=-x_1\ot x_1\wedge x_2\ot v_1 - x_2\ot x_1\wedge x_2\ot v_2
.
\end{aligned}
$$
}\end{example}

\subsection*{Outline}

In \cref{background}, we recall various properties of reflection groups and hyperplane arrangements and relate derivations and differential forms to differential derivations.
We give Saito criteria for invariant derivations and 1-forms
in~\cref{saitocriterion}.
We then derive a Saito criterion for invariant differential derivations for all finite groups
using an extensive analysis of the actions of transvections in~\cref{appendix}.
In Sections \ref{maximalsection}--\ref{slglsection},
we focus on reflection groups whose transvection root spaces are maximal.
We show the hyperplanes all lie
in the same orbit, recall a twisted wedge product, and identify the semi-invariant differential forms in~\cref{maximalsection}.
In~\cref{dualizingsection}, we show how to construct a set of basic 1-forms when a set of basic derivations is known, and vice versa, demonstrating a duality of exponents and coexponents of the group.
The
structure of the set of invariant differential derivations when the characteristic of the base field is not $2$ is given in
\cref{structuresection}
whereas
\cref{char2} analyzes the characteristic $2$ case.
\cref{fpsection} considers groups acting on vector spaces over prime fields and
\cref{slglsection} considers 
$\SL_n(\FF_q)$, $\GL_n(\FF_q)$, 
and all groups in between.
\vspace{1ex}

\section{Background and Notation}
\label{background}

We fix a finite-dimensional vector space $V=\FF^n$ over a field $\FF$ of arbitrary characteristic, $n\geq 1$. 
Let $S:=S(V^*)$ be
the symmetric algebra of $V^*$
which we identify with the 
polynomial ring $\FF[V]=\FF[x_1,\dots,x_n]$ 
for a basis $x_1, \ldots, x_n$
of $V^*$.
We use $\F(S)$ for the fraction field of $S$.
Let $G\subset \GL(V)$ be a finite group acting on $V$
and consider
the usual dual action
on $V^*$
(given by the inverse transpose of the matrix recording
the action on $V$)
which extends 
to an action on $S$ by automorphisms.
We write 
$a \doteq b$ to indicate $a$ and $b$
are equal up to a scalar
in  $\FF^\times$.
Note that all tensor products are taken over $\FF$.

\subsection*{Invariants}
For any $\FF G$-module $M$, we write
$M^G$ for the invariants in $M$ and
\vspace{1ex}
$$M^G_{\chi}=\{m \in M: 
g(m)=\chi(g) m \text{ for all } g\in G\}
\qquad\text{
for the 
$\chi$-invariants},
\vspace{1ex}
$$
the space
of semi-invariants
with respect to a linear
character $\chi: G \rightarrow \FF^\times$
of $G$.
We write 
$\det={\det}_V:G\rightarrow \FF^\times$ 
for the determinant character
of $G$ acting on $V$.

The space
$S\ot M$ is an $S$-module through multiplication in the first tensor component.
Likewise,
the space of invariants
$(S\ot M)^G$
is an $S^G$-module,
necessarily
of rank
$\dim_{\FF}(M)$ (e.g., see~\cite{Brion} or \cite{BroerChuai}),
and we seek
$S^G$-module bases
when these invariant spaces are free.


\subsection*{Reflections}
Recall that a {\em reflection} on a vector space $V=\FF^n$ is a nonidentity invertible linear transformation that fixes pointwise a subspace of $V$ of codimension 1, called the {\em reflecting hyperplane} of the transformation. 
A {\em reflection group} is a group generated
by reflections, and we assume
all reflection groups are finite.
There are two types of reflections:  diagonalizable 
reflections and {\em transvections}
(nondiagonalizable).
Note that order$(s)$ and
$\characteristic\FF$
are coprime 
and 
$\det(s)$ lies in $\FF^{\cross}$
when $s$ is a diagonalizable reflection,
whereas 
order$(s)=\characteristic\FF$
and $\det(s)=1$
when $s$ is a transvection (see \cite{SmithBook}).

\subsection*{Reflection arrangement of a finite group}
We say a hyperplane $H$ in $V$
is a {\em reflecting hyperplane
of $G$} when there is some
reflection in $G$ about $H$.
We denote
the (possibly empty) collection of all reflecting hyperplanes of $G$
by $\A=\A(G)$ and note that $\A(G)=\A(W)$ for
$W$ the subgroup of $G$ generated by the reflections in $G$.

\subsection*{Pointwise stabilizers
of reflecting hyperplanes}
We denote the pointwise stabilizer
of each reflecting hyperplane $H$
of $G$ by
$G_H=\{g\in G: g|_H=1\}$. 
The transvections in $G_H$ along  with the identity form a normal subgroup $K_H$
of $G$:
$$
K_H=\ker\big(\det:G_H\rightarrow\FF^\times\, \big)
\, . 
$$
We set $e_H=|G_H:K_H|$
and observe that
$G_H=\langle K_H, s_H \rangle
\cong K_H\rtimes \ZZ/e_H\ZZ
\, ,
$ where
$s_H$ is 
a diagonalizable reflection in $G$
about $H$ of maximal order $e_H$ when $e_H\neq 1$
and $s_H=1_G$ when $e_H=1$.

\subsection*{Root vectors}
For each reflecting hyperplane $H$ of $G$, we fix a linear form $\ell_H$ in $V^*$ with  $\ker \ell_H=H$.  Each reflection $s$ in $G$ about $H$ is then
defined by its {\em root vector} $v_s$ 
spanning $\text{Im}(s-1)\subset V$ with respect to $\ell_H$,
see \cite{SmithBook}:
$$
s(v)=v+\ell_H(v)v_s \ \ \text{for all }v\text{ in }V.
$$ 
Note that a reflection $s$ about $H$
is a transvection 
exactly when
its root vector $v_s$ lies {\em in} $H$.

\subsection*{Root spaces}	
The {\em root space} $\R_H$
of a reflecting hyperplane $H$ of $G$ 
is the
$\FF$-vector space spanned by all 
of the root
vectors of the reflections 
in $G$
about $H$.
The {\em transvection root space} of $H$ (see~\cite{HartmannShepler})
is 
the space $\R_H\cap H$
spanned by the root vectors of the transvections in $G$ 
about $H$.
We denote its dimension by $b_H$:
$$
b_H :=\dim_\FF( \R_H\cap H)
=
\dim_\FF\FF\text{-span}\{
v_s: s \text{ is a transvection in $G$ about } H\}
\, .
$$
If the transvection root space of $H$ is all of $H$, i.e., $\R_H\ \cap H= H$, then $b_H=n-1$
and we say the transvection root space is {\em maximal}.
Often all of the transvection root spaces for $G$ are maximal, as is the case,
for example, 
when $\SL_n(\FF_q)\subset G \subset
\GL_n(\FF_q)$.

\subsection*{Arrangement polynomials}
We consider the arrangement-defining polynomial $Q$ in $S$
and polynomials $Q_{\det}$
and $Q(\tilde{\A})$ 
(see \cite{HartmannShepler})
which vanish on 
some reflecting hyperplanes
or are $1$:
$$
Q:=\prod_{H\in\A}\ell_H\, ,
\quad
Q_{\det}:=\prod_{H\in\A}\ell_H^{{}_{}\, e_H-1}\, ,
\quad
\text{and}
\quad
Q(\tilde{\A}):=\prod_{H\in\A}\ell_H^{{}_{}\, e_Hb_H}\, .
$$
These polynomials depend only
upon $G$ up to a
scalar in $\FF^\times$.
Recall 
that 
$Q_{\det}$
divides any
polynomial that is
semi-invariant
with respect
to the linear character $\det=\det_V:W\rightarrow \FF^\times$
of the
subgroup $W$
generated by the reflections in $G$.
In fact (see~\cref{SemiinavariantPolys}, \cite{Stanley}, \cite{Nakajima}, \cite{Smith2006}), 
for 
$Q_{\det^{-1}}=\prod_{H\in\A:e_H\neq 1}\ell_H$,
\begin{equation}
\label{Qdet}
S^W_{\det}=Q_{\det} \, S^W
\quad\text{ and }\ \ 
S^W_{\det^{-1}}=Q_{\det^{-1}} \, S^W
\, .
\end{equation}

\subsection*{Vector space basis
for one hyperplane}
For any reflecting hyperplane $H$ of $G$, we may choose a convenient
basis $v_1,\dots,v_n$ of $V$ with dual basis $x_1,\dots,x_n$ of $V^*$ so that $v_1,\dots,v_{n-1}$ span $H$ and $\ell_H=x_n$. 
In fact, we may choose $v_1,\dots,v_{b_H}$ to be root vectors of transvections $t_1,\dots,t_{b_H}$ in $G$ about $H$
and $v_n\notin H$ to be an eigenvector of $s_H$ with eigenvalue $\lambda\in \FF^\times$
of order $e_H$.
With respect to this basis,
\vspace{1ex}
\begin{equation}
\label{ConvenientBasis}
t_m=\begin{smallpmatrix}
1\vspace{-0.5em}\\
&\ddots\\
&&1&&1\vspace{-0.5em}\\
&&&\ddots\\
&&&&1
\end{smallpmatrix}\leftarrow 
\text{\begin{tiny} $m^{\text{th}}$ row
\end{tiny}}
\quad\text{for }1\leq m\leq b_H\quad\text{ and}
\quad
s_H=\begin{smallpmatrix}
1\vspace{-0.5em}\\
&\ddots\\
&&1\\
&&&\lambda
\end{smallpmatrix}
\vspace{1ex}
\, .
\end{equation} 
Note that $e_H=1$,
$s_H=1_G$, and $\lambda=1$
when $G$ contains no diagonalizable reflections about $H$. 
When $\FF=\FF_p$, $G_H$ is precisely $\langle t_1,\dots,t_{b_H},s_H\rangle$, so $|G_H|=e_H\, p^{b_H}$. In general, however, $G_H$ may contain more transvections about $H$  (see Lemma 2.1 in \cite{Jacobians}).


\begin{example}
\label{2x2example}
{\em
Let $V=\FF_5^2$ with standard basis $v_1,v_2$ and dual basis $x_1,x_2$ of $V^*$. Consider the group $G=\langle t, s, g\rangle$ generated by
$$
t=\begin{smallpmatrix}
1&1\\0&1
\end{smallpmatrix},\quad
s=\begin{smallpmatrix}
1&\ 0\\0&-1
\end{smallpmatrix},\quad
\text{ and }
\quad
g=\begin{smallpmatrix}
-1&\ 0\\
\ 0&2
\end{smallpmatrix}
\, .
$$ 
The subgroup of $G$ generated by its reflections is
$W=\langle t,s\rangle$, which fixes a single hyperplane $H=\ker x_2$.
The transvection root space of $H$ has dimension
$b_H=1$ (i.e., $\mathcal{R}_H \cap H = H$)
and the maximal order of a diagonalizable reflection is $e_H=2$. So $Q=Q_{\det}=x_2$ and $Q(\tilde{\A})=x_2^2$.
}
\end{example}


\subsection*{Derivations and differential forms}
We identify
the set of
$S$-derivations $\text{Der}_S$ on $V$
with $S\ot V$,
identify
the set of
differential forms on $V$
with  $S\ot \wedge V^*$,
and consider
the $S$-module $S\otimes \wedge V^*\otimes V$
of {\em differential derivations} 
on $V$
(otherwise called
{\em mixed forms}, see~\cite{ReinerShepler}):
$$
\begin{aligned}
&S\ot V \,   && \text{ (\em derivations)},
\\
&S\ot \wedge V^* 
\,  &&  \text{ (\em differential forms)},
\\
&S\otimes \wedge V^*\otimes V
\,  && \text{ ({\em differential derivations})}
\, .
\end{aligned}
$$
Consider a basis $v_1,\ldots, v_n$
of $V$ with dual basis
$x_1,\ldots, x_n$ of $V^*$ 
and
a set
$\omega_1,\dots,\omega_n \in S\ot V^*$.
For any
subset $I=\{i_1,\dots,i_k\}$ 
of $[n]=\{1, \ldots, n\}$
with
$i_1<\ldots< i_k$,
we set
\begin{equation}
\label{volumes}
\begin{aligned}
v_I&:=v_{i_1}\wedge\cdots\wedge v_{i_k}
\in \wedge^k V\, ,
\\
x_I&:=x_{i_1}\wedge\cdots\wedge x_{i_k}
\in \Wedge^k V^*\, , 
\\
\omega_I&:=\omega_{i_1}\wedge\cdots\wedge\omega_{i_k}
\in S\ot \Wedge^k V^*\, ,
\end{aligned}
\end{equation}
with
$v_I=1$, 
$x_I=1$,
and $\om_I=1\ot 1$ for
the empty set $I=\varnothing$. 
To indicate subsets of size $k$,
we write $I\in\tbinom{[n]}{k}$
for $I\subset [n]$ with $|I|=k$.
We
denote the volume form on $V$ by
$\volv=v_1\wedge\cdots\wedge v_n\in\wedge^n V$
and the volume form on $V^*$
by $\volx=x_1\wedge\cdots\wedge x_n\in\wedge^n V^*$.

\subsection*{Differential derivations
as a module over the differential forms}
We view the set of differential derivations 
$S\ot \wedge V^* \ot V$
as a module over the set of differential forms
$S\otimes \wedge V^*$ via multiplication in the first two tensor components: 
for $f,f'$ in $S$, $x_I,x'_I$ in $\wedge V^*$, and $v\in V$,
\begin{equation}
\label{modulestructure}    
(f\otimes x_I)(f'\otimes x_I'\otimes v):=
ff'\otimes x_I\wedge x_I'\otimes v.
\end{equation}

\subsection*{Embedding derivations into the
differential derivations}
We embed the set of derivations
into the set
of differential derivations:
$$
    S\otimes V
    \hookrightarrow S\otimes \wedge V^*\otimes V,\quad\quad
    f\otimes v \mapsto  f\otimes 1 \otimes v.
$$
This embedding together with the module structure of  \cref{modulestructure}
allows us to
multiply a differential form
and a derivation 
to construct
a differential derivation,
with the $G$-action preserved:
\vspace{-1.5ex}
$$
\begin{aligned}
(S\ot \wedge V^*)\cross (S\ot V)
&\longrightarrow\ S\ot \wedge V^*\ot V,
\\
(f\otimes x_I)\ 
\cross
(f'\otimes v)
&\longmapsto\ 
ff'\otimes x_I \otimes v\, .
\end{aligned}
$$

\subsection*{Degree and rank}
We assign $\deg x_i=1$
for all $i$
so that
$S=\bigoplus_i S_i$  is graded
by the usual polynomial degree
and $G$ acts by graded automorphisms.
For any $\FF G$-module $M$,
we say the elements of $S_i\ot M$
are {\em homogeneous of polynomial degree} $i$.
We say elements in 
$S\ot \Wedge^kV^*$
and in 
$S\ot \Wedge^k V^* \ot V$
have {\em rank} $k$.
Thus $1$-forms are differential forms
of rank $1$, and
for homogeneous $f$ in $S$,
$I\subset [n]$,
and $v\in V$,
$$\deg(f\otimes x_I\otimes v)=\deg(f)\quad\text{ and }\quad \rank(f\otimes x_I\otimes v)=\rank(x_I)=|I|
\, .
$$ 
For an $\FF G$-module $M$, one may choose a homogeneous
basis of the $S^G$-module $(S\ot M)^G$ 
when free
by the graded Nakayama Lemma
(e.g., see~
\cite[Proposition A.20]{OrlikTerao},
\cite[Corollary 5.2.5]{SmithBook},
or 
\cite[Section 2.10]{CampbellWehlau}) and
polynomial
degrees of basis elements are independent of choice.

\subsection*{Euler derivation}
Recall that the {\em Euler derivation}
$\theta_E:=\sum_{i=1}^n x_i\otimes v_i$
is invariant under any
linear group action.
We use the invariant differential derivation $d\theta_E$ (see \cite{ReinerShepler}):
$$
d\theta_E
=1\otimes x_1\otimes v_1+\dots+1\otimes x_n\otimes v_n
\in (S\otimes V^*\otimes V)^G
\, .
$$


\subsection*{Coefficient matrix}
For any
$\FF G$-module $M$, 
we define
the coefficient matrix
of $\omega_1,\dots,\omega_\ell$ 
in $S\otimes M$ 
with respect to an
ordered basis $z_1,\dots,z_m$ of $M$ as usual by
$$
\coef(\omega_1,\dots,\omega_\ell)
:=
\{ f_{ij} \}_{\substack{1\leq i\leq \ell \\ 1\leq j \leq m\rule{0ex}{1.5ex} }} 
\in M_{\ell\times m}(S)
\, ,
\qquad
\text{where }  
\omega_i
=
\sum_{j=1}^{m} f_{ij}\ot z_j
\text{ for } 1\leq i\leq \ell\, .
$$
For any unordered set $\B\subset S\ot M$ with $|\B|=m$,
the determinant 
$\det\coef(\B)$ 
is defined up to a sign
and is nonzero
precisely when $\B$ is
independent over $\F(S)$.  
Notice that
the coefficient vector
of a
differential derivation 
arising as the product of a differential form 
and a derivation
is just the tensor product
of the respective coefficient vectors:
for any $\omega=\sum_{I}f_I\otimes x_I$ in $S\ot \wedge^k V^*$
and $\theta=\sum_{j=1}^n f'_j \ot v_j$ in $S\ot V$,
$$
\omega\theta
=
\sum_{I,j}f_I f'_j\otimes x_I\otimes v_j
\quad\text{ with }\ 
\coef(\omega\theta)
=
\coef(\omega)\ot \coef(\theta)
=
\begin{pmatrix}
f_{I_1},\dots, f_{I_{m}}
\end{pmatrix}
\otimes
\begin{pmatrix}
f'_1,\dots, f'_n
\end{pmatrix}
$$
with respect to
a fixed ordered basis
$x_I \ot v_j$
of $\wedge^k V^*\ot V$
arising from ordered bases
$v_1, \ldots, v_n$ of $V$
and $x_{I_1},\ldots, x_{I_m}$
of $\wedge^k V^*$
with $m=\tbinom{n}{k}$.
This extends to subsets of differential forms $\B\subset S\ot \Wedge V^*$ and derivations $\B'
\subset S\ot V$:
with the appropriate orderings,
$$
\coef(\omega\theta:\omega\in\B,\theta\in\B')
=
\coef(\B)\otimes \coef(\B')
\, .
$$
This fact immediately implies the following observation
since
$\{ \om_I: I \subset 
[n] \}$
is independent over $\F(S)$
whenever 
$\{\om_1,\ldots, \om_n\}
\subset S\ot V^*$
is independent.

\vspace{1ex}
\begin{lemma}
\label{wholesetindependent}
If 
$\{\theta_1,\dots,\theta_n\}\subset S\ot V$
and 
$\{\omega_1,\dots,\omega_n\}
\subset S\ot V^*$
are both 
$\F(S)$-independent, then so is
$$
\big\{\omega_I\, \theta_j:
I\in\tbinom{[n]}{k},
1\leq j\leq n\big\}\quad
\subset
S\ot \Wedge^k V^* \ot V
\, .
$$
\end{lemma}

\section{Saito/Jacobian Criterion}
\label{saitocriterion}


We consider a finite group $G\subset \GL(V)$
acting on $V=\FF^n$.
We give criteria for finding $S^G$-bases
of invariant derivations $(S\ot V)^G$ and invariant $1$-forms $(S\ot V^*)^G$
before examining invariant differential derivations
$(S\ot \Wedge V^*\ot V)^G$.

\vspace{1ex}
\subsection*{Solomon's Theorem}
Solomon \cite{Solomon} showed that when $G$ is a
reflection
group acting on $V=\CC^n$,
the set of invariant
differential forms
$(S\otimes\wedge V^*)^G$ is a free exterior algebra over $S^G$ generated by $df_1,
\ldots, df_n$
for any polynomials
$f_1, \ldots, f_n$ with
$S^G=\CC[f_1, \ldots, f_n]$:
$$
(S\otimes\wedge V^*)^G=\bigwedge{}_{\rule[0ex]{0ex}{2.2ex} S^G}\{\omega_1,\dots,\omega_n\}
\quad\text{with }\omega_i=df_i \text { for }
1\leq i\leq n
\, .
$$
For a reflection group $G$
acting on $V=\FF^n$ with
$\characteristic \FF$
dividing $|G|$,
the ring of invariant polynomials
$S^G$ may not be a polynomial algebra.
Even when 
$S^G=\FF[f_1, \ldots, f_n]$ for
some
$f_i$ in $S^G$ 
(i.e.,
the action is
{\em coregular}, see \cite{Broer2010}), the
exterior derivatives
$df_i$ do not generate $(S\otimes \wedge V^*)^G$ as an exterior algebra
when $G$ contains
transvections:
Hartmann \cite{Hartmann} showed that the conclusion of Solomon’s Theorem holds if and only if $S^G$ is a polynomial algebra and $G$ contains no transvections.


\subsection*{Saito criteria for invariant
derivations and $\mathbf{1}$-forms}

Criteria for finding bases of invariant derivations and
invariant $1$-forms
under the action of a finite linear group $G$ relies on the pointwise
stabilizer subgroups $G_H$ of each of the reflecting hyperplanes $H\in \A$ of $G$.
Thus we begin with groups fixing a single hyperplane; a more general criteria will follow from  \cite[Theorem~3]{BroerChuai}.

We use the $1$-forms from \cite[Remark~13]{HartmannShepler}
and provide a short direct proof for derivations.
For each reflecting hyperplane $H$ of $G$, recall that $e_H$ is the maximal
order of a diagonalizable reflection in $G$ about $H$ (or $e_H=1$ if none exist)
and $b_H$ is the dimension of the transvection
root space of $H$.
We consider the exterior product of derivations
$\theta_1\wedge\cdots\wedge\theta_n$
in $S\ot \Wedge ^n V$
and of 1-forms $\omega_1\wedge\cdots\wedge\omega_n$ in $S\ot \Wedge^n V^*$.
\vspace{2ex}
\begin{lemma}
\label{BasicDerivationsOneHyperplane}
Suppose a nontrivial
finite
group $G\subset \GL(V)$ 
fixes pointwise
a hyperplane $H=\ker\ell_H$ in $V=\FF^n$.
Then
$(S\ot V)^G$ and $(S\ot V^*)^G$ are free $S^G$-modules, and 
for any
$\theta_1,\ldots, \theta_n$
in $(S\otimes V)^{G}$
and any $\om_1,\ldots, \om_n$
in $(S\otimes V^*)^{G}$,
\begin{itemize}\setlength{\itemindent}{-4ex}
    \item 
$\theta_1,\ldots, \theta_n$\hphantom{.} are an $S^G$-basis of $(S\otimes V)^{G}$\hphantom{x}
if and only if
$
\theta_1\wedge\cdots\wedge \theta_n\doteq\ell_H\, \volv
$, and 
\item
$\om_1,\ldots, \om_n$
are an $S^G$-basis
of $(S\otimes V^*)^{G}$
if and only if
$\om_1\wedge\cdots\wedge \om_n
\doteq\ell_H^{e_Hb_H + e_H-1}\, \volx$.
\end{itemize}

\end{lemma}
\begin{proof}
We exhibit an explicit 
 $S^G$-basis $\theta_1, \ldots, \theta_n$ of $(S\ot V)^G$ with
$\theta_1\wedge\cdots\wedge \theta_n\doteq\ell_H\, \volv$ and
an explicit 
 $S^G$-basis $\om_1, \ldots, \om_n$ of $(S\ot V^*)^G$ with
$\om_1\wedge\cdots\wedge \om_n
\doteq\ell_H^{e_Hb_H + e_H-1}\volx$.
The result then follows from \cite[Theorem~3]{BroerChuai}
(see also \cite[Proposition~6]{BroerChuai}).
We use
the basis
$v_1, \ldots, v_n$ 
of $V$ 
with dual basis $x_1, \ldots, x_n$ of $V^*$
of \cref{ConvenientBasis},
so $\ell_H=x_n$, 
and consider the invariants
$$
\theta_i=
\begin{cases}
1\ot v_i
&\text{for }1\leq i<n,\\
\sum_{j=1}^nx_j\otimes v_j
&\text{for } i=n,
\end{cases}
\ \ \ 
\text{and}\ 
\
\omega_i=
\begin{cases}
x_n^{e_H
}\otimes
x_i-x_ix_n^{e_H-1}\otimes x_n
&\text{for }1\leq i\leq b_H,\\
1\otimes x_i
&\text{for }b_H< i
<n, \\
x_n^{e_H-1}\otimes x_n
&\text{for } i=n
\, .
\end{cases}
$$
Then
$\omega_1\wedge\cdots\wedge\omega_n=x_n^{e_Hb_H+e_H-1}
\, \volx$,
and the $\om_i$
are an $S^G$-basis
of $(S\ot V^*)^G$
by~\cite{HartmannShepler}.

Fix some $\theta=\sum_i h_i\ot v_i$ in $(S\ot V)^G$.  For
$g\neq 1_G$ in $G$ with root vector 
$(\alpha_1,\ldots, \alpha_n)$,
so $g(v_n)=\sum_{i\neq n}
\alpha_i v_i + (1+\alpha_n) v_n$
with $\alpha_n\neq -1$,
we equate polynomial coefficients 
of $\theta$ with those of
$g(\theta)$ 
and conclude that
$$
g(h_i)=
\begin{cases}
h_i-\frac{\alpha_i}{1+\alpha_n}\, h_n
\rule[-2ex]{0ex}{2ex}
&\text{ for } i\neq n\, ,\\
\frac{1}{1+\alpha_n}\, h_n
&\text{ for } i=n\, ,
\end{cases}
\qquad\text{ while }\quad
g(x_i)=
\begin{cases}
x_i-\frac{\alpha_i}{1+\alpha_n}\, x_n
\rule[-2ex]{0ex}{2ex}
&\text{ for } i\neq n\, ,\\
\frac{1}{1+\alpha_n}\, x_n
&\text{ for } i=n\, .
\end{cases}
$$
Note that $\alpha_j\neq 0$
for some $j$ (as $g\neq 1_G$), so  
$h_n\doteq g(h_j)-h_j$
is divisible by $x_n$ 
(see~\cref{sf-f}).
Also,
$g$ fixes
$\frac{h_n}{x_n}$
and
$h_i-\frac{h_n}{x_n}x_i$ for $i\neq n$.
As $g$ was arbitrary,
these polynomials lie in $S^G$,
and 
$$\theta=\tfrac{h_n}{x_n}\theta_n+\sum_{i\neq n} (h_i-\tfrac{h_n}{x_n}x_i)\theta_i$$
lies in the $S^G$-span of the $\theta_i$.
As $\theta_1\wedge\cdots\wedge\theta_n=x_n \volv\neq 0$, 
the $\theta_i$ 
are independent over $\F(S)$,
and thus over $S^G$, and are an $S^G$-basis of $(S\ot V)^G$.
\end{proof}

\vspace{1ex}
\cref{BasicDerivationsOneHyperplane} 
together with \cite[Theorem~3]{BroerChuai}
implies
the following analog
of the classical Saito Criterion
(see \cite[Corollary 6.61, Proposition 6.47]{OrlikTerao})
for all finite linear groups, including those with transvections.

\vspace{2ex}
\begin{theorem}
\label{saitoderivations} 
Consider a finite group $G\subset \GL(V)$ acting on $V=\FF^n$.
\\
For homogeneous
$\theta_1,\ldots, \theta_n$ 
in $(S\otimes V)^G$, 
the following
are equivalent:
\begin{itemize}
\setlength{\itemindent}{0ex}
\item [a)]
$(S\otimes V)^G$ is a free $S^G$-module
with basis $\theta_1,\ldots, \theta_n$.
\item [b)] $\theta_1\wedge\cdots\wedge\theta_n\doteq Q\ \volv$.
\item [c)] $\theta_1,\dots,\theta_n$ are independent over $\F(S)$ and 
$\sum_{i=1}^n\deg\theta_i=\deg Q
=|\A|$.
\end{itemize}
For homogeneous
$\omega_1,\ldots, \om_n$ 
in $(S\otimes V^*)^G$, 
the following are equivalent:
\begin{itemize}
\setlength{\itemindent}{0ex}
\item [a)] 
$(S\otimes V^*)^G$  
is a free $S^G$-module
with basis
$\omega_1,\ldots, \om_n$.

\item [b)] $\omega_1\wedge\cdots\wedge\omega_n\doteq Q(\tilde{\A})Q_{\det}\, \volx$.

\item [c)] $\omega_1,\dots,\omega_n$ are 
independent over $\F(S)$ and 
$
 \sum_{i=1}^n\deg\omega_i=\sum_{H\in\A}\big(e_Hb_H+e_H-1\big)$.
\end{itemize}
\end{theorem}

%
\noindent We call $\theta_1,\ldots, \theta_n$
satisfying the equivalent conditions
of the last theorem
{\em basic derivations}
and call
$\omega_1,\ldots, \om_n$ 
satisfying the equivalent conditions
of the last theorem
{\em basic $1$-forms}.


\subsection*{Saito criterion for 
invariant differential derivations}
Now we turn our attention to establishing a Saito criterion for $(S\ot \wedge^k V^*\otimes V)^G$.
This requires a detailed analysis of the action of transvections relegated to \cref{appendix}.
Such care is not required
over fields of characteristic zero as
all reflections are diagonalizable.
Within the analysis,
we distinguish those hyperplanes of $G$
whose pointwise stabilizers $G_H$ consist
of exactly one transvection and the identity. 
Define $\delta_H$ to be $1$ in this case and $0$ otherwise: 
\begin{equation}
\label{deltaH}
    \delta_H:=\left\{\begin{array}{ll}
    1 & \text{if } G_H=\{1_G,\, 
    \text{one transvection}\} \,
        , \\
    0 & \text{otherwise}\,.
\end{array}\right.
\end{equation}
Note that when $\characteristic \FF\neq 2$, any transvection and its inverse are distinct, so 
$$ 
\delta_H = 0 \text{ for all }H\in \A
\quad\text{ whenever char}\, \FF\neq 2
\, .
$$
Additionally, for each
$0\leq k\leq n$ corresponding to the rank of a differential derivation,
we define a polynomial 
which depends only upon $G$ up to a scalar in $\FF^\times$,
\begin{equation}
\label{Q_kdef}
Q_k:=\prod_{H\in\A}\ell_H^{\, e_H\, a_{H,k}}\, ,
\end{equation}
in terms of integers $a_{H,k}\geq 0$
depending on the pointwise stabilizer
$G_H$ of each $H$ in $\A$:
\begin{equation}
\label{adef}
a_{H,k}
:=
(n-\delta_H)\Big(\tbinom{n-1}{k}-\tbinom{n-b_H-1}{k}\Big)
+\tbinom{n-1}{k-1}-\tbinom{n-b_H-1}{k-1}\, .
\end{equation}
Here, $\tbinom{a}{b}=0$ if $a<b$ or $b<0$.

\begin{remark}{\em 
In the nonmodular setting,
the group $G$ contains no transvections and $b_H=0$ for every reflecting hyperplane 
(minimal transvection root spaces), so each $a_{H,k}=0$ and
$$
Q_k= 1
\quad \text{ when }
\characteristic \FF\text{ and }|G|\text{ are coprime}\, .
$$
On the other end
of the spectrum, 
we will see in \cref{maximalsection}
that if $b_H=n-1$ for every reflecting hyperplane of $G$ (maximal transvection root spaces), then the reflecting
hyperplanes are in a single $G$-orbit
and there are fixed nonnegative integers 
$e$, $b$, $\delta$, and $a_k$ with
$e=e_H$, $b=b_H$, $\delta=\delta_H$, and $a_k=a_{H,k}$ 
for every reflecting hyperplane $H$, and
$$Q_k=(QQ_{\det})^{a_k} \quad\text{ with }\quad a_k=
\begin{cases}
0
&\text{when }k=0\, ,\\
(n-\delta)(n-1)
&\text{when }k=1\, ,\\
(n-\delta)\tbinom{n-1}{k}+\tbinom{n-1}{k-1}
&\text{when }k\geq 2\, .
\end{cases}
$$ 
}
\end{remark}

\vspace{2ex}

Now we establish a polynomial that divides the determinant of the coefficient matrix of any potential basis of
invariant differential derivations
of fixed rank.
Compare with \cite[Lemma~6.1]{ReinerShepler}
in the nonmodular case, where $Q_k= 1$ for all $k\geq 0$.
The analysis required here 
(relegated to the appendix) is more nuanced because of the existence
of transvections.

\vspace{1ex}
\begin{lemma}
\label{saitolemma}
Consider a finite group $G\subset \GL(V)$.
For any set $\B$ of $n\tbinom{n}{k}$ elements in $(S\otimes \wedge^k V^*\otimes V)^G$, the determinant of $\coef(\mathcal{B})$ is divisible by the polynomial $$
Q_{\rule[0ex]{0ex}{1ex} 
}^{\sbinom{n-1}{k}} Q_{\det}^{(n-1)\sbinom{n-1}{k-1}} Q_k
\, .
$$
\vspace{-3ex}
\end{lemma}
\begin{proof}
Fix a reflecting hyperplane $H=\ker\ell_H$ of $G$. 
By \cref{detcoefdiv} in \cref{appendix},
$\det\coef(\mathcal{B})$
is divisible by
$\ell_H$ to the power 
$
\tbinom{n-1}{k}+(e_H-1)(n-1)\tbinom{n-1}{k-1}+e_Ha_{H,k}$.
As the linear forms
$\ell_H$ are pairwise coprime
for $H\in \A$,
the claim follows.

\end{proof}


We establish a Saito criterion 
for invariant
differential derivations
in \cref{saitotheorem},
and the proof
depends
on an analysis for pointwise stabilizers $G_H$
of reflecting hyperplanes $H\in \A$.
As for the derivations and differential forms, 
we first require a criterion for the case of a group
fixing one hyperplane.
Recall that we write
$I\in\tbinom{[n]}{k}$
when
$I\subset [n]=\{1,\ldots, n\}$
with $|I|=k$.

\vspace{2ex}
\begin{proposition}
\label{OneHyperplane}
Suppose a nontrivial finite
group $G\subset \GL(V)$ 
fixes pointwise
a hyperplane $H=\ker \ell_H$ in $V=\FF^n$.
Then $(S\ot \wedge^k V^*\ot V)^G$ is a free $S^G$-module for all $k$,
and elements
$\eta_1\ldots, \eta_m$
are a basis
if and only if  $m=n\tbinom{n}{k}$ and
$$
\det\coef(\eta_1,\ldots, \eta_m)\doteq \ell_H^{\sbinom{n-1}{k} + (e_H-1)(n-1)\sbinom{n-1}{k-1}+e_Ha_{H,k}}
\, . $$
\end{proposition}

\begin{proof}
We abbreviate
$\ell=\ell_H$, $e=e_H$, $b=b_H$,  $\delta=\delta_H$, $a_k=a_{H,k}$,
and use the basis $v_1,\dots,v_n$ of $V$ with dual basis $x_1,\dots,x_n$ of $V^*$ of \cref{ConvenientBasis} 
so that $\ell=x_n$
and $x_n^e$ is $G$-invariant.
We also use the basic derivations
$\theta_i$ and basic $1$-forms
$\om_i$ from
the proof of \cref{BasicDerivationsOneHyperplane}.
For a fixed $k$,
consider the subset
of invariant differential derivations
$$
\B_k=\big\{\tilde{\omega}_I\, 
\theta_j:
I\in\tbinom{[n]}{k},\, 
1\leq j\leq n
\text{ with } 
n\notin I\text{ or } j\neq n\big\}
\cup
\big\{\tilde{\omega}_I\, d\theta_E:
I\in\tbinom{[n]}{k-1}
\text{ with } 
n\notin I
\big\}
\, ,
$$
where
$
\tilde{\omega}_I=
\omega_I/x_n^{e\, m(I)}
\text{ for }
m(I)=
\max\big\{0\, ,\, |I\cap\{1,\dots,b,n\}|-1\big\}
$.
By \cref{wholesetindependent}, 
the set 
$\{\tilde{\omega}_I\, \theta_j:I\in\tbinom{[n]}{k},1\leq j\leq n\}$ is independent over $\F(S)$.
We argue that, for each $I$ with $n\notin I$,
we may replace
$\tilde{\omega}_{I\cup\{n\}}\theta_n$ 
in this set
by $\tilde{\omega}_I\, d\theta_E$
while maintaining $\F(S)$-independence
to show
that the resulting set $\B_k$ is also  $\F(S)$-independent.
Note that 
$$
x_n^e\, d\theta_E=\sum_{i=1}^{b}\omega_i\theta_i+\sum_{i=b+1}^{n-1}(x_n^{e}\omega_i\theta_i-x_i\omega_n\theta_i)+\omega_n\theta_n
\, ,
$$
and thus,
for $I\subset\{1,\dots,n-1\}$,
$$
\begin{aligned}
\tilde{\omega}_I\, d\theta_E
&=
\frac{1}{x_n^e}
\sum_{i=1}^{b}
(\tilde{\omega}_I\wedge
\omega_i)\,\theta_i
+
\frac{1}{x_n^e}
\sum_{i=b+1}^{n-1}
\big(
x_n^{e}\,
\tilde{\omega}_I\wedge
\omega_i-x_i\,\tilde{\omega}_I\wedge\omega_n
\big)\,\theta_i
+
\frac{1}{x_n^e}
(\tilde{\omega}_I\wedge\omega_n)\,\theta_n
\, .
\end{aligned}
$$
Thus each $\tilde{\omega}_I\, d\theta_E$ lies in the $\F(S)$-span of 
$\{\tilde{\omega}_I\,\theta_j:I\in\tbinom{[n]}{k},1\leq j\leq n\}$ with the coefficient
of $\tilde{\omega}_{I\cup\{n\}}\theta_n$ nonzero when $n\notin I$.
As the various sets $I\cup\{n\}$ with $n\notin I$
are distinct,
$\B_k$ is $\F(S)$-independent.

First suppose $\delta=0$
and set
$\Delta_k =\tbinom{n-1}{k}
+(e-1)(n-1)\tbinom{n-1}{k-1}+ea_{k}$. 
The module
$(S\otimes \wedge^k V^*\otimes V)^{G}$
is free over $S^G$ by \cite[Proposition~6]{BroerChuai}, 
say
with basis $\mathcal{C}_k$.
Each element of $\B_k$
lies in the $S^G$-span of $\mathcal{C}_k$
with polynomial coefficients recorded by some matrix $M$,
and 
$$
\det\text{Coef}({\mathcal{C}_k})
\cdot
\det({M})\
=
\det\text{Coef}({\B_k})\neq 0
\, .
$$
By \cref{saitolemma},
$\ell^{\Delta_k}$
divides
$\det\coef(\mathcal{C}_k)$ in $S$ (as $\delta=0$), while
a calculation confirms that  $\deg(\det\coef(\B_k))=\Delta_k$.
Hence
$\det\coef(\mathcal{C}_k)
\doteq \ell^{\Delta_k}$
and
\cite[Theorem~3]{BroerChuai}
implies the result.

Now suppose $\delta=1$ and set
$\Delta_k'=\tbinom{n-2}{k}+n\tbinom{n-2}{k-1}+\tbinom{n-2}{k-2}$. Here,
$G=\{1_G,t_1\}$, $\characteristic\FF=2$, and
$e=b=1$. Consider an alternate subset of invariant differential derivations
$$
\begin{aligned}
\B_k'=
&\big\{\tilde{\omega}_I\, 
\theta_j:
I\in\tbinom{[n]}{k},\, 
1\leq j\leq n
\text{ with } 
I\cap\{1,n\}=\varnothing \text{ or } j\neq n\big\}
\\
&\cup
\big\{\tilde{\omega}_I\, d\theta_E:
I\in\tbinom{[n]}{k-1}
\text{ with }
n\notin I
\big\}
\cup
\{\tilde{\omega}_I\, \eta_0:I\in\tbinom{[n]}{k-1}\text{ with } I\cap\{1,n\} =\varnothing
\big\}
\, ,
\end{aligned}
$$
where $\eta_0=x_1\ot x_1\ot v_1 + x_n\ot x_1\ot v_n + x_1\ot x_n\ot v_1 + x_1\ot x_n\ot v_n$,
which is $G$-invariant.
We argue that, 
for each $I$ with $I\cap\{1,n\} =\varnothing$,
we may replace
$\tilde{\omega}_{I\cup\{1\}}\theta_n$ in $\B_k$
by $\tilde{\omega}_I\,\eta_0$ while maintaining $\F(S)$-independence to show that $\B_k'$ is also $\F(S)$-independent. 
As $\characteristic\FF=2$,
$$
x_n\, \eta_0=
\sum_{i=2}^{n-1}x_i\omega_1\theta_i+(x_1^2+x_1x_n)\omega_n\theta_1+\omega_1\theta_n\, ,
$$
and hence, for $I\subset\{2,\dots,n-1\}$,
$$
\tilde{\omega}_I\, \eta_0
=
\sum_{i=2}^{n-1}\Big(\frac{\tilde{\omega}_I\wedge\omega_1}{x_n}\Big)\theta_i
+(x_1^2+x_1x_n)\Big(\frac{\tilde{\omega}_I\wedge\omega_n}{x_n}\Big)\theta_1
+\Big(\frac{\tilde{\omega}_I\wedge\omega_1}{x_n}\Big)\theta_n
\, .
$$
Thus 
each 
$\tilde{\omega}_I\, \eta_0$ 
with 
$I\cap\{1,n\}=\varnothing$
lies
in the $\F(S)$-span of 
a subset of
$\B_k$
with nonzero coefficient of $\tilde{\omega}_{I\cup\{1\}}\theta_n$.
As these subsets
are disjoint for the various
$I$ with $I\cap\{1,n\}=\varnothing$,
the set
$\B_k'$ is $\F(S)$-independent
and $\det\coef(\B'_k)\neq 0$.
A computation shows that 
$\Delta_k'$ is simultaneously
the degree 
of $\det\coef(\B_k')$
and the degree of the polynomial
in \cref{saitolemma} (as $\delta=1$).
The claim then follows as in the previous case
using \cite[Theorem~3]{BroerChuai}.
\end{proof}

\cref{OneHyperplane}
with \cite[Theorem~3]{BroerChuai}
then implies \cref{SaitoThmIntro}
of the introduction:

\vspace{1ex}
\begin{theorem}
\label{saitotheorem}

Consider a finite group $G\subset \GL(V)$ acting on $V=\FF^n$.
For a set $\B$ of $n\tbinom{n}{k}$ homogeneous elements in $(S\ot \wedge^k V^*\ot V)^G$, the following are equivalent:

\begin{enumerate}
\setlength{\itemindent}{-1ex}
\item[a)]
$(S\otimes \wedge^k V^*\otimes V)^G$ is a free $S^G$-module with basis $\B$.

\item [b)] $\det \coef(\B) \doteq  Q_{\rule[0ex]{0ex}{1ex} 
}^{\sbinom{n-1}{k}} Q_{\det}^{(n-1)\sbinom{n-1}{k-1}} Q_k$.
\rule[-2ex]{0ex}{6ex}

\item [c)]
$\B$ is independent over $\F(S)$
and 
$\displaystyle
\sum_{\eta\in\B} \deg\eta=\sum_{H\in\A} \tbinom{n-1}{k}+(e_H-1)(n-1)\tbinom{n-1}{k-1}+e_H a_{H,k}\, .
$
\end{enumerate}
\end{theorem}

\begin{example}
\label{dim1example}
{\em
Let $G\subset \GL(V)$ be a nontrivial finite group with $\dim V=1$.
Then $G$ is cyclic
say with generator $s$ of order $e>1$.
Notice that $s$ is a reflection fixing the hyperplane
$H=\{0_V\}$, so $G$ is a reflection group.
All elements in $G$ are diagonalizable, 
so $G$ does not contain any transvections, 
and the transvection root space of $H$
has dimension $0=n-1$ (so is maximal).
Let $v$ be a basis of $V$ with dual basis $x$ of $V^*$.
Then $\omega=x^{e-1}\ot x$ 
is an $S^G$-basis of $(S\ot V^*)^G$
and 
$\theta=x\ot v$ 
is an $S^G$-basis of $(S\ot V)^G$ (see \cref{saitoderivations}).
Finally, $\theta$ and $d\theta_E$ are an $S^G$-basis of $(S\ot \wedge V^*\ot V)^G$ by \cref{saitotheorem}.
Note that $\om$ and $\theta$
are dual in some sense,
see \cref{dim1remark}.
} 
\end{example}

\begin{remark}{\em 
Note that one direction
of \cref{saitotheorem}
follows directly
from a version of
Solomon's 1963 original
argument \cite{Solomon}.
Indeed,
(b) and (c) 
are equivalent by 
\cref{saitolemma}, and
any $\F(S)$-independent ordered subset
$\B=\{\eta_1,\ldots, \eta_m\}$
with 
$m=n\binom{n}{k}$
spans $\F(S)\otimes M$
over $\F(S)$
for $M=\wedge^kV^*\otimes V$.
Hence, after relabeling the basis elements $x_I\otimes v_j$ of $M$
as $z_1,\dots,z_m$,
we may
write a fixed $\eta\in(S\otimes M)^G$ as
$$
\eta=\sum_j f_j\otimes z_j
\quad\text{ and as}\quad
\eta=\sum_i h_i\, \eta_i
\text{ for some $h_i\in \F(S)$}.
$$
Here,
$\eta_i=\sum_j
\coef(\mathcal{B})_{ij}\otimes z_j$
(with each 
$\coef(\B)_{ij}$
in $S$)
and 
$\textbf{h}\cdot \begin{matrix}
\coef(\mathcal{B})
\end{matrix}
=\textbf{f}$
for row vectors $\textbf{h}=(h_1,\dots,h_m)$ and $\textbf{f}=(f_1,\dots,f_m)$.
Cramer's Rule implies that
$h_i
=\frac{\det\coef(\mathcal{B})_{(i)}}{\det\coef(\mathcal{B})}$,
where $\coef(\mathcal{B})_{(i)}$ is obtained by replacing the $i$-th row of $\coef(\mathcal{B})$ by $\textbf{f}$. 
Since 
$\eta_1\wedge\cdots\wedge\eta_{m}
=\det\coef(\B)\ot z_1
\wedge\cdots\wedge z_m$,
the polynomial
$\det\coef(\B)$
is
semi-invariant with respect to the linear character
${\det}_M^{-1}$ of $G$
for $\det_M$ the character
afforded by $\Wedge^m M$,
as is $\det \coef(\B)_{(i)}$ likewise,
and hence $h_i\in \F(S)^G$. 
By \cref{saitolemma},
$\det\coef(\B)$ divides
 $\det\coef(\B)_{(i)}$,
so $h_i$ lies in $S\cap \F(S)^G=S^G$ (see, e.g., \cite[Section 1.7]{CampbellWehlau} or \cite[Section 1.2]{SmithBook}).
}
\end{remark}


\begin{remark}
{\em 
Of particular interest is
the set 
$(S\ot V^*\ot V)^G$ of invariant
differential derivations of
rank $1$
(see \cite{ParkingSpaces, ReinerShepler, ReinerSommersShepler}).
By \cref{saitotheorem}, $n^2$ homogeneous
$\F(S)$-independent elements in $(S\ot V^*\ot V)^G$ are an $S^G$-basis if and only if
their polynomial degrees
add to
$$
\Delta_1=\sum_{H\in\A}e_H(b_Hn-b_H\delta_H+n-1) 
\, .
$$
}
\end{remark}


\begin{example}
{\em
Consider the group
$$
G=\left\langle
\begin{smallpmatrix}
1&1&0\\
0&1&0\\
0&0&1
\end{smallpmatrix},
\begin{smallpmatrix}
1&0&1\\
0&1&0\\
0&0&1
\end{smallpmatrix},
\begin{smallpmatrix}
1&0&0\\
0&1&1\\
0&0&1
\end{smallpmatrix},
\begin{smallpmatrix}
1&0&0\\
0&-1&0\\
0&0&-1
\end{smallpmatrix}
\right\rangle
\subset \GL_3(\FF_3)
\, .
$$
The subgroup $W$ generated by the reflections of $G$ 
is the group
of unipotent upper triangular matrices in $\GL_3(\FF_3)$
and
$\A=\A(G)=\A(W)$
is 
defined by $Q=x_2^3x_3-x_2x_3^3$.
Three hyperplanes $H$ in $\A$ each have transvection root space of dimension $b_H=1$, whereas $b_H=2$ for one hyperplane
($\ker x_3$).
Note that $e_H=1$ and $\delta_H=0$ for all $H$ in $\A$.
The ring of $W$-invariant polynomials
is $S^W=\FF[f_1, f_2, f_3]$ for
$$
f_1=x_3,\ \ 
f_2=x_2^3-x_2x_3^2,\ \ 
f_3=x_1^9 - x_1^3x_2^6 - x_1^3x_2^4x_3^2 - x_1^3x_2^2x_3^4 - x_1^3x_3^6 + x_1x_2^6x_3^2 + x_1x_2^4x_3^4  + x_1x_2^2x_3^6\, .
$$
Additionally, $(S\ot V)^W$ and $(S\ot V^*)^W$ are free $S^W$-modules with respective bases $\theta_1,\theta_2,\theta_3$ and
$\omega_1,\omega_2,\omega_3$ by \cref{saitoderivations} for
$$
\begin{aligned}
\theta_1&=1\ot v_1,\quad
\theta_2=x_1\ot v_1 + x_2\ot v_2 + x_3\ot v_3,\quad
\theta_3=x_1^3\ot v_1 + x_2^3\ot v_2 + x_3^3\ot v_3,\\
\omega_1&=1\ot x_3,
\quad
\omega_2=x_3\ot x_2 -x_2\ot x_3,\\
\omega_3&=(x_2^3x_3-x_2x_3^3)\ot x_1+(-x_1^3x_3+x_1x_3^3)\ot x_2
+ (x_1^3x_2-x_1x_2^3)\ot x_3\, .
\end{aligned}
$$
By \cref{saitotheorem},
$(S\ot V^*\ot V)^W$ is a free $S^W$-module with basis
$$
\{d\theta_E\}\ \cup\
\{\omega_i\theta_j:1\leq i,j\leq 3\}\setminus\{\omega_3\theta_1\}
\, .
$$
Here,
$(x_2^3x_3-x_2x_3^3)\, d\theta_E=\omega_3\theta_1+\omega_2\theta_3-x_3^2\omega_2\theta_2+(x_2^3-x_2x_3^2)\omega_1\theta_2$,
so we may indeed replace $\om_3\theta_1$
(or alternatively $\omega_2\theta_3$) by $d\theta_E$
in the set from
\cref{wholesetindependent} (with $k=1$)
to obtain a basis.

Notice that although
$S^G$ is not a polynomial ring 
(as $G$ is not a reflection group, see \cite{Serre}),
the derivations
$\theta_1,\theta_2,\theta_3$ 
lie in $(S\ot V)^G$
so are a basis of the $S^G$-module $(S\ot V)^G$
by \cref{saitoderivations}.
However, $(S\ot V^*)^G$ can not be free
over $S^G$ 
since otherwise a basis
would also serve as an
$S^W$-basis of $(S\ot V^*)^W$
by \cref{saitoderivations}
(see \cite[Corollary~6]{BroerChuai}) and thus would contain
an element of polynomial degree $0$,
which is not possible as
$(S_0\ot V^*)^G$ is empty.
Note that
$\omega_i\in (S\otimes V^*)^G$
only for $i=2,3$.
Similarly, $(S\ot V^*\ot V)^G$ is not a free $S^G$-module by \cref{saitotheorem}
as $\dim_{\FF}(S_0\ot V^*\ot V)^G$ is $1$,
not $2$ (as for $W$).
} 
\end{example}

\section{Groups with Transvection Root Spaces Maximal}
\label{maximalsection}

We now consider groups whose transvection root spaces are maximal, i.e., groups for which
each transvection root space
coincides with
its reflecting hyperplane
(so $b_H=n-1$ for
all $H$ in the reflection arrangement $\A$).
Such groups include the special and general linear groups $\SL_n(\FF_q)$ and $\GL_n(\FF_q)$ (see \cref{slglsection}). We 
show that all reflecting hyperplanes
are in one orbit and recall a twisted
wedging that exhibits the invariant differential forms as a free
exterior algebra.
We also consider semi-invariant differential forms
with respect to a linear character.

\subsection*{Only one orbit of reflecting hyperplanes}
Recall that a group $G
\subset{\GL(V)}$ acts on its set
$\A$
of reflecting hyperplanes in $V=\FF^n$
with $gH=H'$ for $g$ in $G$ whenever
a reflection in $G$ about $H\in \A$
is conjugate by $g$
to a reflection in $G$ about $H'\in \A$.

\begin{proposition}
\label{sameorbit}
Let $G\subset \GL(V)$ be a finite group.
Any two reflecting hyperplanes 
of $G$ with maximal transvection root spaces lie in the same orbit.
\end{proposition}
\begin{proof}
Fix two such hyperplanes $H=\ker \ell_H$ and $H'=\ker \ell_{H'}$.
Since the transvection root space of
$H$ is maximal, we may 
choose a root vector $v_1$
of a transvection $t$ in $G$ about $H$
with
 $v_1\notin H'$.
Similarly, we may choose 
 a root vector $v_2\notin H$  
of a transvection $t'$ in $G$ about $H'$. 
Extend $v_1, v_2$ to a basis $v_1,\dots,v_n$ of $V$ (so $v_3,\dots,v_n$ span $H\cap H'$), 
and rescale $v_2$ and $\ell_{H'}$ so that $\ell_H(v_2)=1$ while $v_2$ remains a root vector with respect to $\ell_{H'}$.
Then for the dual basis $x_1, \ldots, x_n$ of $V^*$, $\ell_H=x_2$ and $\ell_{H'}=\alpha x_1$ for some $\alpha$ in $\FF^\times$,
and with this basis of $V$,
$$
t=
\left(
\begin{smallarray}{cc|ccc}
1&1&&\\
0&1&&\rule[-1ex]{0ex}{1ex}
\\
\hline\rule[-2ex]{0ex}{4ex}
&&1\vspace{-1.3em}\\
&&&\ddots\\
&&&&1
\end{smallarray}
\right)
\quad\text{ and }\quad
t'=\left(
\begin{smallarray}{cc|ccc}
1&0&&\\
\alpha &1&&\rule[-1ex]{0ex}{1ex}\\
\hline\rule[-2ex]{0ex}{4ex}
&&1\vspace{-1.3em}\\
&&&\ddots\\
&&&&1
\end{smallarray}
\right)
.
$$
Thus $\langle t,t'\rangle$ 
is isomorphic to a 
finite subgroup of $\SL_2(\FF)$. 
We use the classification
of such groups by Dickson,
see \cite[Chapter~3,~Section~6]{Suzuki} (see also Chapter 2, Theorem 6.8):
\begin{enumerate}
    \item [1)] $p=2$ and $\langle t,t'\rangle \cong D_{2m}$,
    the dihedral group of order $2m$, with $m$ odd,
     or
    \item [2)] $p=3$ and $\langle t,t'\rangle \cong \SL_2(\FF_5)$, or
    \item [3)] $\langle t,t'\rangle \cong\SL_2(\FF_q)$ for some $p$-power $q$,
\end{enumerate}
where $p=\characteristic\FF$.
In the first case, the transvections of $\langle t,t'\rangle$ have order $2$, so correspond to the reflections in $D_{2m}$, which all lie in the same conjugacy class as $m$ is odd.
In the second case, the transvections have order $3$, so similarly correspond to the elements of order $3$ in $\SL_2(\FF_5)$, which again are all in the same conjugacy class.
As $t$ and $t'$ are conjugate
in these two cases, $H$ and $H'$ lie in the same orbit.

In the final case, 
we first notice that
the transvections 
in $K_0:=\SL_2(\FF_q)$ are precisely
the elements of order $p$,
and the same is true for the group 
$K:=\langle t,t'\rangle$
since $K$ fixes 
$v_3, \ldots, v_n$.
Hence the
transvections
$t,t'$ in $K$ 
about hyperplanes
$H,H'$, respectively, in $V$
correspond under the isomorphism
to 
transvections
$t_0,t'_0$ in $K_0$ about some hyperplanes
$H_0,H'_0$, respectively, in $V_0=\FF_q^2$.
But $\SL_2(\FF_q)$ acts transitively on the set of projective points in $\FF_q^2$, i.e., all hyperplanes in $V_0$
are in the same $K_0$-orbit.
Thus the pointwise stabilizer
subgroups 
$\text{Stab}_{K_0}(H_0)$
and 
$\text{Stab}_{K_0}(H_0')$
are conjugate in $K_0$,
which implies 
that the pointwise stabilizer
subgroups 
$\text{Stab}_{K}(H)$
and 
$\text{Stab}_{K}(H')$
are likewise conjugate in $K$.
This follows from the fact
that all of these 
stabilizer subgroups have
a purely group-theoretic description:
the transvections in $K_0$ about a fixed
hyperplane are exactly
the order $p$ elements
that commute with any fixed
transvection about that hyperplane, and the same
is true for $K$.
Thus $H$ and $H'$ lie in the same orbit,
although $t$ and $t'$ may not be conjugate.
\end{proof}

\cref{sameorbit}
has the following immediate implication.
\begin{cor}
\label{actstransitively}
Let $G\subset \GL(V)$ be a finite group 
with transvection root spaces all maximal. Then $G$ acts transitively on the set of its reflecting hyperplanes.
\end{cor}

\begin{remark}
\label{OneOrbit}
{\em
For a finite group $G$ acting
linearly, when
reflecting hyperplanes $H$, $H'$ of $G$ are in the same $G$-orbit,
their pointwise stabilizers $G_H$, $G_{H'}$ are conjugate in $G$. 
Thus \cref{actstransitively} implies that for groups $G$ whose transvection root spaces are all maximal, 
we have nonnegative integers $e$, $b$, $\delta$, and $a_k$ such that
$$
e=e_H,
\ b=b_H,
\ \delta=\delta_H,\text{ and } a_k=a_{H,k}
\quad\text{ for all }
H\in \A.
$$
}
\end{remark}

\subsection*{Twisted wedge product}

We use the {\em twisted wedge product} of \cite{HartmannShepler}
on differential forms
invariant under
the action of 
a reflection group $G$
whose transvection root
spaces are maximal:
for $\omega,\omega'$ in $(S\ot \wedge V^*)^G$,
we set $\omega\curlywedge\omega':= \omega\cdot \omega'$
when rank $\om$ or rank $\om'$ is $0$ and
\begin{equation}
\label{twistedwedgeproduct}
\omega\curlywedge\omega'
\ :=\
\frac{\omega\wedge\omega'}{Q^e
\rule{0ex}{1.5ex}}
\quad\text{ when rank }
\om,\om' \geq 1 
\, ,
\end{equation}
for $e=e_H$
for all $H$ in $\A$ 
(see \cref{actstransitively}
and \cref{OneOrbit}).
Here, we use the fact that
$Q^e$ in $S$ divides 
$\omega \wedge \om'$ 
for all $\omega$, $\omega'$
in $(S\ot \wedge V^*)^G$
of rank at least $1$ (see \cite{HartmannShepler}).
For $\om_1,\ldots, \om_n$ in $S\otimes V^*$ 
and nonempty
$I=\{ i_1,\ldots, i_k\}\subset[n]$ with
$i_1< \ldots< i_k$,
we set
\begin{equation}
\label{curlywedgeomega}
\omega^{\curlywedgeup}_I
:=\omega_{i_1}\curlywedge\cdots\curlywedge\omega_{i_k}= \frac{\omega_{i_1}\wedge\cdots\wedge \omega_{i_k}}{Q^{\,e(k-1)}}
=\frac{\om_I}{Q^{\, e(k-1)}}
\, 
\end{equation}
and 
$\om^{\curlywedgeup}_I
=\om_{i_1}\curlywedge\cdots\curlywedge \om_{i_k}
:=1\ot 1$ for $k=0$, $I=\varnothing$. 

We next use \cref{saitoderivations}
to slightly strengthen Theorem 10 of \cite{HartmannShepler},
noting that the arguments in its proof
hold even when $S^G$ is not a polynomial ring. 
We use the 
free exterior algebra
from \cite{HartmannShepler}
$$
\bigcurlywedge_{S^G}
\{ \om_1,\ldots, \om_n\}
=
S^G\!\text{-span}
\{ \om_{i_1}\curlywedge\dots\curlywedge
\om_{i_k}:
1\leq i_1<\ldots<i_k\leq n
\}
=
S^G\!\text{-span}
\{\om_I^{\curlywedgeup}: I\subset[n]\} .
$$
\vspace{-3ex}

\begin{theorem}
\label{twistedalgebra}
Let $G\subset \GL(V)$ be a reflection group
with transvection root spaces all maximal. 
Suppose $(S\ot V^*)^G$ is a free $S^G$-module
with basic $1$-forms $\omega_1,\dots,\omega_n$.
Then
$\omega_1,\dots,\omega_n$ generate $(S\otimes \wedge V^*)^G$ 
as a free exterior algebra over $S^G$ via the twisted wedge product:
$$
(S\otimes \wedge V^*)^G
=
\bigcurlywedge_{S^G}
\{ \om_1,\ldots, \om_n\}
\, .
$$
\end{theorem}
\begin{remark}
\label{q-Catalan}{\em 
Results of \cite{HartmannShepler}
and \cref{twistedalgebra} 
suggest an analog
of the {\em $\mathsf{q}$-Catalan number} 
for
groups $G$ with $\SL_n(\FF_q)\subset G \subset \GL_n(\FF_q)$ 
(see \cref{slglsection}) and more generally 
reflection groups 
$G\subset \GL_n(\FF)$
with maximal transvection root spaces.
For such groups with exponents $m_1, \ldots, m_n$ and $S^G=\FF[f_1, \ldots, f_n]$ for $f_i$ homogeneous
of degree $d_i$,
$$
{\rm Hilb}
\big( 
(S\ot \Wedge V^*)^G, 
\mathsf{q},\mathsf{t} \big)
\ =\ 
\frac{ (1-\mathsf{q}^{e|\A|})
+ \mathsf{ q}^{e|\A|}\prod_{i=1}^n
(1
+\mathsf{ q}^{m_i-e|\A|}{
\mathsf{t}})
\rule[-1ex]{0ex}{2ex}}
{ \prod_{i=1}^n 
( 1 - {\mathsf{ q}}^{d_i} )
\rule{0ex}{2ex}}
\, 
$$
is the Hilbert series of the bigraded $\FF$-vector space
of invariant differential forms,
$$
\text{Hilb}
\big( (S\ot \Wedge V^*)^G,
\mathsf{q}, \mathsf{t}\big)\
:=
\sum_{i\geq 0,\ k\geq 0}
\dim_{\FF}
(S_i\ot \wedge^k V^*)^G 
\
\mathsf{q}^i\, \mathsf{ t}^k\, .
$$
For a real reflection group $G$
with Coxeter number $h$, one takes $\mathsf{t}=-\mathsf{q}^{h+1}$
to recover the $\mathsf{q}$-Catalan number for $G$ (see~\cite{BerestEtingofGinzburg2003,Gordon2003}).
For an extension to complex reflection groups, see
~\cite{GordonGriffeth}
(and also \cite{Stump}). Here
one might consider
$h=e|\A|$ (see
 \cref{CoxeterNumber}).
 }
\end{remark}
\subsection*{Semi-invariant differential forms}
Let $\chi:G\rightarrow \FF^\times$
be a linear character of a reflection group $G$.
Then $\chi$ must be the identity on all transvections in $G$
as they have order $p=\characteristic\FF$. 
This implies that 
Stanley's argument \cite{Stanley} (see also \cite{Nakajima},\cite{Smith2006})
for reflection groups acting over $\CC$
extends to actions
over $\FF$ to show that
\begin{equation}
    \label{SemiinavariantPolys}
S^G_{\chi}=Q_{\chi}\, S^G
\, 
\end{equation}
for the polynomial
$Q_{\chi}:=\prod_{H\in\A}\ell_H^{\, c_H}$,
where $c_H$ is 
the smallest nonnegative integer such that $\chi(s_H)=\det^{-c_H}(s_H)$ for each $H\in \A$.
Here, $s_H$ again is a diagonalizable reflection about $H$ of maximal order $e_H$ when $e_H>1$ and the identity otherwise.
Thus for $k=0$, 
$(S\ot \Wedge^k V^*)^G_{\chi}=S^G\, (Q_{\chi}\ot 1)$.  We give the structure
for $k>0$ next.

\begin{proposition}
\label{DetInverseInvariantForms}
Let $G\subset\GL(V)$ be a reflection group whose transvection root
spaces are all maximal. 
Then
$(S\ot \Wedge^k V^*)^G
\subset 
Q_{\chi^{-1}} (S\ot \Wedge^k V^*)$
for 
any linear character $\chi: G\rightarrow
\FF^\times$
and
$$
(S\ot \Wedge^k V^*)^G_{\chi}
=
\frac{1}{Q_{\chi^{-1}}}(S\ot \Wedge ^k V^*)^G
\, 
\quad\text{ for }
k>0\, .
$$
\end{proposition}
\begin{proof}
Fix $\om\in(S\ot \Wedge^k V^*)^G$
and let $H=\ker \ell_H$ be a reflecting hyperplane
of $G$.
We use the basis $v_1,\dots,v_n$ of $V$ and $x_1,\dots,x_n$ of $V^*$ of \cref{ConvenientBasis} and write $\om=\sum_I f_I\ot x_I$ for some $f_I\in S$.
By \cite[Lemma~4]{HartmannShepler}, $f_I$ is divisible by $\ell_H^{\, e_H}$ when $n\notin I$ since $b_H=n-1$ and $k>0$.
Additionally, the last equation in the proof shows that $f_I$ is divisible by
$\ell_H^{\,e_H-1}$ when $n\in I$.
So 
$\omega$ lies in
$\ell_H^{\,e_H-1}S\ot \Wedge V^*$
and, as $H$ was arbitrary,
in $Q_{\det}S\ot \Wedge V^*$.
But
 $Q_{\chi^{-1}}$ divides $Q_{\det}$,
so $\om$ lies in
$Q_{\chi^{-1}}S\ot \Wedge V^*$.
This implies that $(S\ot \wedge^k V^*)^G_{\chi}
\supset (S\ot \wedge^k V^*)^G/ Q_{\chi^{-1}}$.
The reverse inclusion follows from the fact that
$Q_{\chi^{-1}}(S\ot \Wedge^k V^*)^G_{\chi}
\subset
(S\ot \Wedge ^k V^*)^G$.
\end{proof}

\section{Dualizing Derivations and $1$-forms}
\label{dualizingsection}

In this section, 
we show that 
$(S\ot V)^G$ is free
if and only if
$(S\ot V^*)^G$ is free 
over $S^G$
for a reflection
group $G\subset\GL(V)$ 
acting on $V=\FF^n$
with transvection root spaces
maximal
and give
a duality between exponents
and coexponents.
Indeed,
we show how to construct dual invariant 
$1$-forms from invariant derivations and vice versa.
This in turn 
gives a condition for the invariant differential
forms to be generated as a
free exterior algebra
under the twisted wedge product.
We apply these
results to $\GL_n(\FF_q)$
and $\SL_n(\FF_q)$
in \cref{slglsection}.

\subsection*{Perfect pairing}
We use a familiar perfect pairing 
between $\wedge^k V$ and 
$\wedge^{\! n-k}\, V^*
\cong(\wedge^{\! n-k}\,V)^*$ related to the Hodge star operator:
let $\Phi$ be the isomorphism
$$
\begin{aligned}
\Phi: \wedge^k V  &\rightarrow 
\wedge^{n-k}\, V^*,
\quad
\beta&\mapsto \
\Big(\beta' \mapsto \frac{\beta'\wedge \beta}{\volv}
\Big)\, ,
\end{aligned}
$$
where $\volv=v_1\wedge\cdots\wedge v_n$
is the volume form
for a fixed choice of basis $v_1, \ldots, v_n$ of $V$.

\begin{remark}
\label{alternatingsigns}
{\em
Let $x_1,\dots,x_n$ be the basis of $V^*$ dual to the basis $v_1,\dots,v_n$ of $V$. 
For a subset $I=\{i_1,\dots,i_k\}$ 
of $[n]=\{1,\dots,n\}$,
set $I^c=[n]\backslash I$, the complementary subset.
Then
$\Phi(v_I)= x_{I^c}$ (see \cref{volumes})
up to a sign 
given by the Levi-Civita symbol.
In particular, 
$$
\Phi(v_1\wedge\cdots\wedge \widehat{v_j}\wedge\cdots\wedge v_n)=(-1)^{j+1}\, x_j
\, .
$$
}
\end{remark}

\subsection*{Dualizing map is semi-invariant}
We extend $\Phi$ to a function 
$\wedge V\rightarrow \wedge V^*$.
Note that $\Phi$ is a
skew $\GL(V)$-homomorphism
with respect
to the character $\det=\det_V$
of $\GL(V)$:
\begin{lemma}
\label{skewrelationship}
The dualizing map $\Phi$ is
$(\det^{-1})$-invariant:
for any $g$ in $\GL(V)$,
$$
g(\Phi)={\det}^{-1}(g) \, \Phi\, ,
\quad{i.e., }\ \ 
g\circ \Phi ={\det}^{-1}(g)\ \Phi\circ g
\, .
$$
\end{lemma}
\begin{proof}
Fix $k$ and recall that $g(\volv)=\det(g)\volv$. 
For any $\beta\in\wedge^k V$ and $\beta'\in\wedge^{n-k} V$, 
$$
\big((g\Phi)(\beta)\big)(\beta')
=
\big(g(\Phi( g^{-1}\beta))\big)(\beta')
=
g\big(\Phi( g^{-1}\beta)(g^{-1}\beta')\big)
=g\Big(\frac{g^{-1}\beta'\wedge g^{-1}\beta}{\volv}\Big)
=\frac{\beta'\wedge \beta}{\det(g)\volv}\, ,
$$
which is just
${\det}^{-1}(g)\, \Phi(\beta)(\beta')
$.
\end{proof}

\subsection*{Dual 1-forms}
We use this perfect pairing to construct $G$-invariant $1$-forms from $G$-invariant derivations via the linear map
$$
1\ot \Phi: S\otimes \wedge V \rightarrow S\otimes \wedge V^*
\, .
$$
\begin{proposition}
\label{dualderivations}
Consider a reflection group
$G\subset \GL(V)$.
Suppose $(S\otimes V)^G$ is a free $S^G$-module 
with basic derivations
$\theta_1,\dots,\theta_n$.
For each $1\leq i\leq n$, define
dual 1-forms
$$
\omega_i=\theta_i^*:=
(1\otimes \Phi) \Big(Q_{\det}\,\theta_1 \wedge \cdots \wedge \widehat{\theta}_{i} \wedge \cdots \wedge \theta_n\Big)
\quad \in S\otimes V^*
\, .
$$ 
Then $\omega_1\dots,\omega_n$
are a basis of $(S\ot V^*)^G$ as a free $S^G$-module 
if and only if
the transvection root spaces 
of $G$ are all maximal,
in which case
$\om_1,\dots,\om_n$
also generate 
$(S\ot \wedge V^*)^G$
as a free exterior algebra over $S^G$ 
via the twisted wedge product of \cref{twistedwedgeproduct}.
\end{proposition}
\begin{proof}
Since $G$ is a reflection group,
$Q_{\det}$ is $\det$-invariant (see \cref{Qdet}),
and hence
each $\om_i$ is indeed
invariant by \cref{skewrelationship}
as the $\theta_i$ are invariant.
We assume without loss of generality that the $\theta_i$ are homogeneous.
Let $A$ be the coefficient matrix $\coef(\theta_1, \ldots, \theta_n)$
so $\det A \doteq Q$ by \cref{saitoderivations}.
The determinant of the minor matrix $A_{ij}$ is precisely the
polynomial coefficient of $v_1\wedge\cdots\wedge 
\widehat{v}_j\wedge\cdots\wedge v_n$ in $\theta_1\wedge\cdots\wedge\widehat{\theta}_i\wedge\cdots\wedge\theta_n$.
We replace each $\omega_i$ by $(-1)^{i+1}\, \omega_i$ so that the sign changes coincide with those for the cofactors $c_{ij}$ in the cofactor matrix $C=(\det A) A^{-t}$ of $A$ (see \cref{alternatingsigns}). Then
$$
\text{each} \ \ 
\omega_i
= 
\sum_{j=1}^n
Q_{\det}\, c_{ij}\otimes x_j
\qquad\text{ and }\ \ 
\coef(\om_1,\ldots,\om_n)
=
Q_{\det}\, C
\doteq
Q\, Q_{\det}\, A^{-t}
\, .
$$
Thus
$\det\coef(\om_1\ldots, \om_n)\doteq Q^n Q_{\det}^n \det A^{-t}
\doteq Q^{n-1}\, Q_{\det}^n$,
which equals $ Q(\tilde{\A})\,Q_{\det}$ 
exactly when the transvection roots
spaces are maximal, i.e., 
$b_H=n-1$ for all $H\in\A$. The claim then follows from \cref{saitoderivations,twistedalgebra}.
\end{proof}

\subsection*{Dual derivations}
Alternatively,
one may use $\Phi^{-1}$ to
define dual derivations.
We provide a brief proof in the same style as 
that for \cref{dualderivations}.

\begin{proposition}
\label{alternateduality}
Consider a reflection group $G\subset \GL(V)$
for $\dim V=n>1$
with transvection root spaces all maximal.
Suppose $(S\ot V^*)^G$ is a free $S^G$-module
with basic $1$-forms
$\omega_1,\dots,\omega_n$.
For each $1\leq i\leq n$, define dual derivations
$$
\theta_i=\omega_i^*:=
(1\ot\Phi^{-1})
\Big(\frac{
\omega_1\curlywedge\cdots\curlywedge\widehat{\omega_i}\curlywedge\cdots\curlywedge\omega_n}
{Q_{\det}}
\Big)
\, .
$$
Then $\theta_1,\dots,\theta_n$ are a basis of $(S\ot V)^G$ as a free $S^G$-module.
\end{proposition}

\begin{proof}
By  
\cref{DetInverseInvariantForms}, $Q_{\det}$
divides 
in $S\ot \Wedge V^*$
the indicated twisted wedge
product 
and the quotient
is $(\det^{-1})$-invariant.
By \cref{skewrelationship}, the map $\Phi^{-1}$
is itself $\det$-invariant and thus
each $\theta_i$ lies
$(S\ot V)^G$.
Consider the coefficient matrix
$A=\coef(\om_1,\ldots,\omega_n)$ and the cofactor matrix $C=\{c_{ij}\}=(\det A)A^{-t}$ of $A$.
We assume without loss of generality that the $\om_i$ are homogeneous.
Then
$\det A \doteq
Q^{en-1}$ by \cref{saitoderivations}
(as $b_H=n-1$ and $e=e_H$ for 
all $H\in \A$)
and,
after replacing $\theta_i$ by $(-1)^{i+1}\, \theta_i$ to match the sign changes of $C$,
\vspace{1ex}
$$
\text{each} \ \ 
\theta_i
=\sum_{j=1}^n \frac{c_{ij}}{Q^{e(n-2)}\,Q_{\det}}
\quad
\text{ and }\quad
\coef(\theta_1,\ldots,\theta_n)
=
\frac{C}{Q^{e(n-2)}\,Q_{\det}}
\doteq Q^eA^{-t}\, .
\vspace{1ex}
$$
The claim follows from \cref{saitoderivations}
since
$\det\coef(\theta_1,\ldots, \theta_n)= Q^{en}\det A^{-t}\doteq Q$.
\end{proof}

\begin{remark}
\label{dim1remark}
{\em
A version of
\cref{alternateduality}
also holds when $n=1$
provided we again apply
$(1\ot \Phi^{-1})$ to 
a generator of 
$(S\ot \wedge^{n-1}V^*)^G_{\det^{-1}}$ 
(see \cref{DetInverseInvariantForms}).
For $n=1$, we use the bases 
$v$ of $V$ and $x$ of $V^*$
as well as
the derivation $\theta$ and $1$-form $\om$ from \cref{dim1example}.
Here, 
$(S\ot \wedge^{n-1}V^*)^G_{\det^{-1}}
=S^G(Q\ot 1)$ 
since $Q=Q_{\det^{-1}}=x$,
see \cref{Qdet} or \cref{SemiinavariantPolys}, so we apply
$(1\ot\Phi^{-1})$
to $Q\ot 1$ to dualize $\om$:
\vspace{1ex}
$$\om^*=(1\ot\Phi^{-1})(Q\ot 1)
=Q\ot v=x\ot v =\theta
\, .
\vspace{1ex}
$$
Note that 
the dual 
of $\theta$ here by \cref{dualderivations}
is
just 
\vspace{1ex}
$$
\theta^*=(1\ot \Phi)(Q_{\det}\ot 1)=Q_{\det}\ot x=\om
\quad
\text{ for }
Q_{\det}=x^{e-1}
\, .
\vspace{1ex}
$$
} 
\end{remark}

\subsection*{Duality of exponents and coexponents}

\cref{dualderivations,alternateduality} imply an analog of the
duality of {\em exponents}
and {\em coexponents}
(see \cite{OrlikTerao})
for
well-generated complex reflection groups.
Recall that for any finite group $G$, if $(S\ot V^*)^G$
is a free $S^G$-module,
the set of
polynomial degrees in a 
homogeneous basis does not
depend on choice of basis,
and likewise for $(S\ot V)^G$.
\begin{cor}
\label{DualityOfExponents}
Let $G\subset \GL(V)$ be 
a reflection group 
with transvection root spaces all maximal.
Then $(S\ot V)^G$ is a free $S^G$-module
if and only if
$(S\ot V^*)^G$ is a free $S^G$-module.
When both modules are free with respective homogeneous bases
of polynomial degrees
$m_1^*, \ldots, m_n^*$
and
$m_1, \ldots, m_n$,
then
\vspace{1ex}
$$m_i^* + m_i = e|\A|
\, ,
\vspace{1ex}$$
after possibly
reindexing,
where
$|\A|$ is the number
of reflecting hyperplanes of $G$
and $e$ is the maximal order of a
diagonalizable reflection in $G$.
\end{cor}
\begin{proof}
Suppose $\theta_1, \ldots, \theta_n$
are a homogeneous $S^G$-basis of $(S\ot V)^G$.
Then the
dual $1$-forms $\om_1, \ldots, \om_n$ afforded by \cref{dualderivations}
give an $S^G$-basis of $(S\ot V^*)^G$ with
\vspace{1ex}
$$
m_i=\deg\om_i=\deg Q_{\det} + \sum_{j\neq i }
\deg \theta_j = \deg Q_{\det} + \deg Q
-\deg \theta_i
=e\deg Q-m_i^*
=e|\A|-m_i^*
\, ,
\vspace{1ex}
$$
since $\omega_i$ is dual to $\theta_i$ and $Q_{\det} Q =Q^e$.
Alternatively, if the $1$-forms $\om_1, \ldots, \om_n$
are an $S^G$-basis of $(S\ot V^*)^G$,
then the dual derivations 
$\theta_1, \ldots, \theta_n$
are an $S^G$-basis for $(S\ot V)^G$
by \cref{alternateduality}
(use \cref{dim1remark}
when $n=1$)
with again $m_i+m_i^*=e|\A|$.
\end{proof}

\begin{remark}\label{CoxeterNumber}
{\em 
Recall that the exponents $m_i$ and coexponents $m_i^*$ of a {\em duality (well-generated) complex
reflection group}
$G\subset \GL_n(\CC)$
(e.g. any Weyl or Coxeter group)
may be ordered so 
$$
m_i + m_i^* 
= 
\text{Coxeter number}
=
\deg f_n
=
\frac{\deg QJ}{n}
= 
\frac{\deg Q Q_{\det} Q(\tilde \A)}{n}
\, 
$$
as $\sum_i m_i^*=\deg Q$
and $\sum_i m_i=\deg J$
for the
Jacobian determinant 
$J=\det \{ 
\tfrac{\del f_i}{\del x_j} \}
\doteq
Q_{\det}
=Q_{\det}Q(\tilde{\A})$
as $G$ contains no transvections.
Here,
$f_1, \ldots, f_n$
are homogeneous
basic invariants
for $G$ ordered with nondecreasing degrees.
One thus may be tempted 
by \cref{DualityOfExponents}
to
regard the integer 
$$e |\A|
= 
(\text{maximal order of a diagonalizable
reflection})
 \cdot
(\# \text{reflecting hyperplanes})
$$
as the {\em Coxeter number}
of a reflection group $G\subset\GL_n(\FF)$ for arbitrary $\FF$ with
transvection root spaces all
maximal.
In this case,
we use
$QQ_{\det} Q(\tilde{\A})$
in favor of the 
{\em discriminant}
$Q
J$
and note that
$$
\begin{aligned}
e |\A|
&= \frac{n e|\A|}{n}
=\frac{|\A| + (e-1)|\A| +
(n-1)e|\A|}{n}
= 
\frac{\deg Q Q_{\det} Q(\tilde \A)}{n}
\, .
\end{aligned}
$$
Reflection
groups with transvection root spaces all maximal 
thus may serve as
modular
analogues
of the 
duality (well-generated)
complex reflection groups (also see \cref{slglsection}).
}
\end{remark}

\section{Structure Theorem for Invariant Differential Derivations}
\label{structuresection}
We investigate the structure of $(S\ot \wedge V^*\ot V)^G$
when $G\subset\GL(V)$ 
is a reflection group with 
transvection root spaces
maximal.
Such is the case when $\SL_n(\FF_q)\subset G \subset \GL_n(\FF_q)$
(see \cref{slglsection})
or $G$ is
the pointwise stabilizer in $\SL_n(\FF_q)$ or $\GL_n(\FF_q)$ of a hyperplane in $V$,
for example.
Recall that
$(S\ot V)^G$ is free if and only if
$(S\ot V^*)^G$ is free over $S^G$ in this setting
(see \cref{DualityOfExponents}).
We start with basic derivations in
$(S\ot V)^G$
and use the dual $1$-forms in
$(S\ot V^*)^G$ afforded by \cref{dualderivations}
to construct an $S^G$-basis
for $(S\ot \wedge V^*\ot V)^G$.
Alternatively, we 
could instead construct
the same $S^G$-basis
for $(S\ot \wedge V^*\ot V)^G$
starting with
basic $1$-forms in $(S\ot V^*)^G$
and using the dual derivations in
$(S\ot V)^G$ 
from \cref{alternateduality}.
 
We suppose $\characteristic\FF\neq 2$ in this section and save the $\characteristic\FF=2$ case for \cref{char2}.
Recall that $V=\FF^n$ and
we write $I\in\tbinom{[n]}{k}$ when $I\subset[n]=\{1,\dots,n\}$ with $|I|=k$.

\begin{lemma}
\label{basisindependent}
Consider a reflection group $G\subset \GL(V)$ with transvection root spaces all maximal.
Suppose $(S\ot V)^G$ is a free $S^G$-module with basic derivations $\theta_1,\dots,\theta_n$ 
and dual 1-forms $\omega_1,\dots,\omega_n$. 
Then the following
two subsets
of $(S\otimes \wedge V^*\otimes V)^G$
of size $m=n\tbinom{n}{k}$
are $\F(S)$-independent
for any $r=1,\dots, n$:
\begin{small}
\begin{itemize}\setlength{\itemindent}{-4ex}
\item
$
\big\{d\theta_E\big\}
\cup
\big\{\omega_I^{\curlywedgeup}\, \theta_1,\dots,\omega_I^{\curlywedgeup}\, \theta_n:
I\subset[n]\big\}
\setminus 
\big\{\om_r \theta_r \big\}
\rule[-2ex]{0ex}{2ex}$
\item
$
\big\{
\omega^{\curlywedgeup}_I\, \theta_1,
\ldots,
\omega^{\curlywedgeup}_I\, \theta_{r-1},
\omega^{\curlywedgeup}_I\, \theta_{r+1},
\ldots,
\omega^{\curlywedgeup}_I\, \theta_n\! :
I\subset[n],\,
r\in I \}
\cup
\big\{\omega^{\curlywedgeup}_I\, d\theta_E,
\omega^{\curlywedgeup}_I\,\theta_1,
\dots,
\omega^{\curlywedgeup}_I\,\theta_n\!\! :
I\subset[n],
r\notin I\big\}
\, .
$
\end{itemize}
\end{small}
\end{lemma}

\begin{proof}
Fix $r$ and note that the given forms indeed all lie in 
$(S\ot \wedge V^*\ot V)^G$ (see \cref{twistedwedgeproduct} and \cref{dualderivations}). 
For each $k$, 
denote the collection of elements of rank $k$ in the first set in the claim
by $\B_k$ and in the second set  by
$\B_k'$.
When $k=0$,  $\B_k=\B_k'=\{\theta_1,\dots,\theta_n\}$,
which is independent over $\F(S)$ as it is a basis of $(S\ot V)^G$.
Note that \cref{wholesetindependent} implies that 
$\mathcal{C}_k=\{\omega^{\curlywedgeup}_I\,\theta_j:
I\in\tbinom{[n]}{k},\, 
1\leq j\leq n\}$ is independent over $\F(S)$ for all $k$. 
As $\B_k=\mathcal{C}_k$
for $k\geq 2$
and $\B_1=\B_1'$,
it is left to show
that, for $k\geq 1$,
$$
\B_k'=\big\{\omega^{\curlywedgeup}_I\,\theta_j:I\in\tbinom{[n]}{k}
\text{ with }
r\notin I
\text{ or }
j\neq r\big\}
\cup
\big\{\omega^{\curlywedgeup}_I\, d\theta_E : I\in\tbinom{[n]}{k-1}
\text{ with }
r\notin I
\big\}
$$
is independent over 
$\F(S)$.
Recall from the proof of \cref{dualderivations} that
$$
\coef(\omega_1,\dots,\omega_n)
= Q_{\det} \, C
\doteq
Q \, Q_{\det}\, A^{-t}
\quad 
\text{ for }
\coef(\theta_1,\dots,\theta_n)=A, 
$$
after replacing $\omega_i$ by $(-1)^{i+1}\omega_i$, where 
$C
=\{c_{ij}\}
=(\det A)\, A^{-t}$ 
is the cofactor matrix of
$A=\{a_{ij}\}$,
as $\det A\doteq Q$ by \cref{saitoderivations}.
Then
$\omega_m
=Q_{\det}\sum_i c_{mi}\otimes x_i$
and $\theta_m=\sum_{j}a_{mj}\otimes v_j$
for $1\leq m\leq n$
and,
as $C^t\, A=(\det A)\, I$,
$$
\sum_{m=1}^n \omega_m\, \theta_m
=Q_{\det}
\sum_{i,j} 
\Big(\sum_{m=1}^n
c_{mi}
a_{mj}
\Big)
\otimes x_i\otimes v_j
\doteq Q_{\det}\sum_i Q\otimes x_i\otimes v_i
=Q^e\, d\theta_E
\, .
$$
Hence, for any $I\subset\{1,\dots,n\}$,
$$
\begin{aligned}
\omega^{\curlywedgeup}_{I}
\, d\theta_E
& \doteq
\frac{1}{Q^e\rule{0ex}{1.5ex}
}
\sum_{m=1}^n (\omega^{\curlywedgeup}_I\wedge \omega_m)\, \theta_m
=
\sum_{m\notin I} 
\Big(\frac{\omega^{\curlywedgeup}_I\wedge \omega_m}{Q^e}\Big)\, \theta_m 
\, .
\end{aligned}
$$
Thus each
$\omega_I^{\curlywedgeup}\, d\theta_E$ lies in the $\F(S)$-span of $\mathcal{C}_k$ with nonzero coefficient of $\omega_{I\cup\{r\}}^{\curlywedgeup}\theta_r$
when $r\notin I$.
As the various sets $I\cup\{r\}$ with $r\notin I$ are distinct, $\B_k'$ is $\F(S)$-independent for $k\geq 1$.
\end{proof}

Now that we have  $\F(S)$-independent
sets of the appropriate size, we show that they each yield a basis of $(S\ot \wedge V^*\ot V)^G$.
We obtain
\cref{AllButOneInto}
of the introduction,
again 
using the dual $1$-forms of  \cref{dualderivations}:

\begin{theorem}
\label{AllButOne}
Let $G\subset \GL(V)$ be a reflection group with transvection root spaces all maximal 
and $\characteristic \FF\neq 2$.
Suppose $(S\ot V)^G$ is a free $S^G$-module with basic derivations $\theta_1,\ldots, \theta_n$ and dual 1-forms $\om_1,\ldots, \om_n$. 
Then $(S\ot \wedge V^*\ot V)^G$ is a free $S^G$-module with basis
$$
\big\{d\theta_E\big\}
\cup
\big\{\omega_I^{\curlywedgeup}\, \theta_1,\dots,\omega_I^{\curlywedgeup}\, \theta_n:
I\subset[n]\big\}
\setminus 
\big\{\om_r \theta_r \big\}
\quad
\text{for any  $r=1,\dots, n$}.
$$
\end{theorem}
\begin{proof}
Fix $r$ and assume without loss of generality that the $\theta_i$ are homogeneous.
For each $k$, let $\B_k$ 
be the collection of 
elements in the proposed
basis of rank $k$.
Then $\B_k$ is an $\F(S)$-independent 
subset of
$(S\ot\wedge^k V^*\ot V)^G$ by \cref{basisindependent}.
It suffices to show by \cref{saitotheorem} that $\sum_{\eta\in\B_k}\deg \eta=\Delta_k$ for
$$
\begin{aligned}
\Delta_k
=\begin{cases}
\deg Q
&\text{when } k=0\, ,
\\
(en^2-e)\deg Q
&\text{when }k=1\, ,
\\
\tbinom{n}{k}(en-k+1)\deg Q
&\text{when }k\geq 2
\, .
\end{cases}
\end{aligned}
$$
When $k=0$,  $\B_k=\{\theta_1,\dots,\theta_n\}$, so indeed
$\sum_{\eta\in \B_k}\deg \eta 
= \deg Q 
= \Delta_0$ by \cref{saitoderivations}.

When $k=1$,
$\B_k=\{d\theta_E\}\cup\{\omega_i\theta_j:(i,j)\neq(r,r)\}$, thus
$$
\sum_{\eta\in \B_1}\deg \eta
=\!\!\!
\sum_{1\leq i,j\leq n}\hspace{-0.1in} \deg \omega_i\theta_j - \deg \omega_r\theta_r + \deg d\theta_E
=
n\sum_i\deg\omega_i+n\sum_j \deg \theta_j- \deg\omega_r\theta_r+\deg d\theta_E
,
$$
which is 
$(en^2-e)\deg Q=\Delta_1$
since
$\sum_i\deg \omega_i
=(en-1)\deg Q$ and
$\sum_j\deg\theta_j=\deg Q$
by \cref{saitoderivations} (see \cref{dualderivations}),
$\deg d\theta_E=0$,
and (see the proof of \cref{DualityOfExponents})
$$
\deg\omega_r\theta_r=
\deg\omega_r + \deg\theta_r
=e|\A|
=e\deg Q
\, .
$$

When $k\geq 2$, 
$$
\sum_{\eta\in \B_k}\deg \eta
=
\sum_{I\in\tbinom{[n]}{k},\\ 1\leq j\leq n} 
\deg \omega^{\curlywedgeup}_I\, \theta_j\, 
=
n\sum_{I\in\tbinom{[n]}{k}}\deg\omega^{\curlywedgeup}_I + \tbinom{n}{k} \sum_j\deg\theta_j
\, .
$$
As $\deg \omega^{\curlywedgeup}_I
=\sum_{i\in I}\deg \omega_i-e(k-1)\deg Q$
(see \cref{twistedwedgeproduct}) and $n\tbinom{n-1}{k-1}=k\tbinom{n}{k}$, this is
$$
n \tbinom{n-1}{k-1} \sum_i \deg\omega_i - n \tbinom{n}{k} e (k-1)\deg Q
+ \tbinom{n}{k} \sum_j\deg \theta_j
=
\tbinom{n}{k}(en-k+1)\deg Q
=\Delta_k
\, .
\vspace{-5ex}
$$
\end{proof}
	
\begin{remark}{\em 
\cref{AllButOne}
implies that
$\{\omega_I^{\curlywedgeup}\,\theta_j:
I\in\tbinom{[n]}{k},
1\leq j\leq n\}$ is an $S^G$-basis of $(S\otimes \wedge^k V^*\otimes V)^G$
when $k\geq 2$
for a reflection group $G\subset \GL(V)$ with transvection root spaces all maximal and $\characteristic \FF\neq 2$,
provided $(S\ot V)^G$
is free with 
basic derivations $\theta_1,\ldots, \theta_n$
and dual $1$-forms
$\om_1, \ldots, \om_n$.
}\end{remark}

We see in \cref{char2}
that the following corollary
of \cref{AllButOne}
also holds when
$\characteristic \FF =2$.

\begin{cor}
\label{freeimpliesfree}
Suppose $\characteristic \FF\neq 2$.
Let $G\subset \GL(V)$ be a reflection group with transvection root spaces all maximal.
If $(S\ot V)^G$  is a free $S^G$-module, then so is
$(S\ot \wedge V^*\ot V)^G$.
\end{cor}
We give an alternate
basis
from which we derive
a module structure
in \cref{structuretheorem2},
again 
using the dual $1$-forms of  \cref{dualderivations}.
\begin{theorem}
\label{ExplicitBasis}
Let $G\subset \GL(V)$ be a reflection group with transvection root spaces all maximal and $\characteristic \FF\neq 2$. 
Suppose $(S\ot V)^G$ is a free $S^G$-module
with
basic derivations $\theta_1,\ldots, \theta_n$
and
dual 1-forms $\om_1,\ldots, \om_n$.
Then $(S\otimes \wedge V^*\otimes V)^G$ is a free $S^G$-module with basis
\vspace{1ex}
\begin{small}
$$
\big\{
\omega^{\curlywedgeup}_I\, \theta_1,
\ldots,
{\omega^{\curlywedgeup}_I\, \theta_{r-1}},
{\omega^{\curlywedgeup}_I\, \theta_{r+1}},
\ldots,
\omega^{\curlywedgeup}_I\, \theta_{n}:
I\subset[n],\,
r\in I \}
\cup
\big\{\omega^{\curlywedgeup}_I\, d\theta_E,
\omega^{\curlywedgeup}_I\,\theta_1,
\dots,
\omega^{\curlywedgeup}_I\,\theta_n:
I\subset[n],
r\notin I\big\}
\vspace{1ex}
$$
\end{small}
for any $r=1,\dots, n$.
\end{theorem}
\begin{proof}
Fix $r$ and assume without loss of generality that the $\theta_i$ are homogeneous.
For each $k$, let $\B_k'$ 
be the collection of 
elements in the proposed
basis of rank $k$.
Then $\B_k'$ is an $\F(S)$-independent 
subset of
$(S\ot\wedge^k V^*\ot V)^G$ by \cref{basisindependent}.
We argue that
$\sum_{\eta\in\B_k'}\deg\eta$
coincides with
$\sum_{\eta\in\B_k}\deg\eta$,
where $\B_k$ is the
collection of elements 
in the basis afforded
by \cref{AllButOne}
of rank $k$ for this fixed $r$.
It will follow then from
\cref{saitotheorem}
that $\B_k'$ is also a basis.

Notice $\B_k=\B_k'$ for $k=0,1$,
so assume $k\geq 2$.
Recall that $\deg\om_r\theta_r=e\deg Q$
(see the proof of \cref{AllButOne}).
Then for nonempty $ I\subset\{1,\dots,n\}$ with $r\notin I$,
$$
\deg \omega_{I\cup\{r\}}\theta_r
=
\deg \om_I^{\curlywedgeup}
\curlywedge\om_r\, \theta_r
= 
\deg\om_I^{\curlywedgeup} + \deg \om_r\theta_r
-e\deg Q
=
\deg \om_I^{\curlywedgeup}
= \deg \om_{I}^{\curlywedgeup}\,  d\theta_E
\, .
$$ 
Thus, as $k\geq 2$, the elements in $\B_k$
not in $\B_k'$ have the same polynomials degrees
as the elements in $\B_k'$ not in $\B_k$, 
and thus
$
\sum_{\eta\in \B'_k}
\deg \eta
=\sum_{\eta\in \B_k} \deg \eta
$.
\end{proof}

\begin{remark}
{\em
Recall again that when the transvection
root spaces of a reflection group $G$ are all maximal,
$(S\ot V)^G$ is free over $S^G$ if and only if
$(S\ot V^*)^G$ is free over $S^G$ 
(see \cref{DualityOfExponents}).
For \cref{AllButOne}
and \cref{ExplicitBasis},
rather than assuming 
that $(S\ot V)^G$ is free
and using the dual $1$-forms,
we may instead assume $(S\ot V^*)^G$ is free 
and use the dual derivations 
of \cref{alternateduality} (see \cref{dim1remark} for $n=1$) 
to obtain the same $S^G$-bases of $(S\ot \wedge V^*\ot V)^G$. 
Indeed, the proofs simply rely on the fact that
$\om_i$ and $\theta_i$ are dual.
}
\end{remark}

\subsection*{Module structure
over twisted subalgebra}
For a well-generated complex reflection group $G$,
$(S\ot \wedge V^*\ot V)^G$ is a direct sum of submodules of rank $1$ over 
$$
\bigwedge_{S^G}\{\om_1,\ldots, \om_{n-1}
\}
$$ 
for homogeneous
basic $1$-forms
$\om_1, \ldots, \om_n$ 
with $\om_n$ of maximal polynomial degree
(see \cite[Theorem~1.1]{ReinerShepler}).
One asks if a similar result holds over arbitrary fields $\FF$
for reflection groups whose transvection root spaces are maximal. 
\cref{ExplicitBasis} implies a more subtle decomposition 
with three key differences from the characteristic zero setting:
here we may omit any one of the basic $1$-forms 
in constructing a suitable subalgebra of invariant differential forms,
we use 
the twisted wedge product of \cref{curlywedgeomega}
instead of the regular wedge product,
and we require an ideal of invariant differential forms.
We define for $r=1, \ldots, n$
$$
A_r:=\bigcurlywedge_{S^G}
\{\omega_1,\dots,
\widehat{\om_r}, \ldots, \omega_{n}\}
=S^G\text{-span}\{\omega_I^{\curlywedgeup}:I\subset[n],r\notin I\}
\quad
\subset
(S\ot \wedge V^*)^G
$$
and use the ideal generated by $\om_r$ under $\curlywedge$:
$$
A_r\curlywedge \om_r
:=
S^G\text{-span}\{ \om_I^{\curlywedgeup} : I\subset[n]
\text{ with }
r\in I\}
\, .
$$

The following corollary of
\cref{ExplicitBasis} provides a module structure over $A_r$.

\begin{cor}
\label{structuretheorem2}
Let $G\subset\GL(V)$ be a reflection group 
with transvection root spaces all maximal and $\characteristic \FF\neq 2$. 
Suppose $(S\ot V)^G$ is a free $S^G$-module
with basic derivations $\theta_1,\ldots, \theta_n$ and dual 1-forms $\om_1,\ldots, \om_n$.
Then, for any $r=1,\dots,n$,
$(S\otimes \wedge V^*\otimes V)^G$
is a direct sum of $A_r$-submodules:
$$
(S\otimes \wedge V^*\otimes V)^G
=
\bigoplus_{j=1}^{n}
A_r\, \theta_j 
\oplus
\bigoplus_{j\neq r} \, (A_r\curlywedge\om_r)\, \theta_j
\oplus
A_r\, d\theta_E
\, .
$$ 
\end{cor}

\subsection*{Hilbert series}
We consider the Hilbert series
of the bigraded $\FF$-vector space
of invariant differential derivations:
$$
{\rm Hilb}
\big( (S\ot \wedge V^* \ot V )^G,
\mathsf{ q}, \mathsf {t}\big)\
:=
\sum_{i\geq 0,\ k\geq 0}
\dim_{\FF}
(S_i\ot \wedge^k V^*\ot V)^G 
\
{\mathsf q}^i\, {\mathsf t}^k\, .
$$
For a Coxeter group $G$,
this Hilbert series gives
the first Kirkman number
(see \cite{Rhoades2014, ParkingSpaces, BerestEtingofGinzburg2003, ReinerSommersShepler}).

\begin{cor}
\label{HilbertSeries}
Let $G\subset \GL(V)$ be a reflection group 
with transvection root spaces all maximal
and $\characteristic \FF\neq 2$.
If $(S\ot V)^G$ is a free $S^G$-module
with homogeneous
generators of polynomial degrees $m_1^*, \ldots, m_n^*$, 
then
$$
\begin{aligned}
{\rm Hilb}
&\big( (S\ot \wedge V^* \ot V )^G,
\mathsf{q}, \mathsf{ t}\big)\\
 & =
{\rm Hilb}\big(S^G, {\mathsf q}\big)
\Big(
{\mathsf t}-{\mathsf q}^{e|\A|}{\mathsf t}+
\big(\sum_{i=1}^n {\mathsf q}^{m_i^*}\big)\big(
1-{\mathsf q}^{e|\A|}
+{\mathsf q}^{e|\A|}\prod_{i=1}^{n}
(1+{\mathsf q}^{-m_i^*}{\mathsf t})
\big)
\Big)
\, ,
\end{aligned}
$$
where $e$ is the
maximal order of a diagonalizable reflection in $G$ and
$|\A|$ is the number of reflecting
hyperplanes of $G$.
\end{cor}
\begin{proof}
Say $\theta_1,\ldots,
\theta_n$ is 
a homogeneous $S^G$-basis of $(S\ot V)^G$ with $\deg\theta_i=m_i^*$.
Consider the
dual $1$-forms
$\om_1,\ldots, \om_n$
of polynomial degree $m_i=\deg\om_i=e|\A|-m_i^*$ of \cref{dualderivations} (see \cref{DualityOfExponents})
and set
$A
=\medcurlywedge_{\FF}
\{\omega_1,\dots,\omega_{n}\}
=\FF\text{-span}\{\omega_I^{\curlywedgeup} : I\subset [n]\}.$
Then
\cref{AllButOne} (with $r=n$ for
example)
implies the result:
${\rm Hilb}
\big( (S\ot \wedge V^* \ot V )^G, 
\mathsf{q}, \mathsf{t}\big)$
is 
${\rm Hilb}\big( S^G, 
\mathsf{q})$
times
$$\begin{aligned}
&
{\rm Hilb}\Big(\bigoplus_{
I,\, j} \FF \, \om_I^{\curlywedgeup}\, \theta_j,
\mathsf{q},\mathsf{t}\Big)
-
{\rm Hilb}\Big(\FF\, \om_n\, \theta_n,
\mathsf{q},\mathsf{t}\Big)
+
{\rm Hilb}(\FF\, d\theta_E,
\mathsf{q},\mathsf{t})
\, ,
\end{aligned}
$$
and a computation confirms that ${\rm Hilb}
\big(\bigoplus_{
I,\, j} \FF \, \om_I^{\curlywedgeup}\, \theta_j,
\mathsf{q},\mathsf{t}\big)$
is 
$$\begin{aligned}
{\rm Hilb}(A,\mathsf{q}, \mathsf{t})
\cdot
{\rm Hilb}(\oplus_{j}\ \FF\, \theta_j,
\mathsf{q},\mathsf{t})
=
\Big(1-
\mathsf{q}^{e|\A|}
+\mathsf{q}^{e|\A|}\prod_{i=1}^n (1+\mathsf{q}^{m_i-e|\A|}\, \mathsf{t})\Big)
\Big(\sum_{i=1}^n 
\mathsf{q}^{m_i^*}\Big)
\, .
\end{aligned}
$$
\end{proof}

See \cref{slglsection} for examples 
of reflection groups with transvection root spaces
maximal.
Groups fixing a single hyperplane pointwise
provide other examples, see
\cref{OneHyperplane} (with $b_H=n-1$)
and \cref{ConvenientBasis}.
Note that the hypothesis that the transvection root spaces of $G$ are all maximal in \cref{AllButOne} is critical,
as we see in the following example.

\begin{example}{\em 
We consider a reflection group $G$
over $\FF_p$ for a prime
$p>2$
where $S^G$ is not a polynomial algebra
(see \cite[Section~8.2]{CampbellWehlau}):
$$G=\left\langle
\begin{smallpmatrix}
1&0&0&0\\
0&1&0&0\\
1&0&1&0\\
0&0&0&1
\end{smallpmatrix},
\begin{smallpmatrix}
1&0&0&0\\
0&1&0&0\\
0&0&1&0\\
0&1&0&1
\end{smallpmatrix},
\begin{smallpmatrix}
1&0&0&0\\
0&1&0&0\\
1&1&1&0\\
1&1&0&1
\end{smallpmatrix}
\right\rangle=
\left\{
\begin{smallpmatrix}
1&0&0&0\\
0&1&0&0\\
a&c&1&0\\
c&b&0&1
\end{smallpmatrix}:a,b,c\in\FF_p
\right\}
\, .
$$
Here,
$Q=x_1^px_2-x_1x_2^p$
and $e_H=b_H=1$, and $\delta_H=0$ for each reflecting hyperplane $H$ of $G$.
Both
$(S\otimes V)^G$ and $(S\ot V^*)^G$
are free $S^G$-modules with respective bases
$\{\theta_i\}$ and $\{\omega_i\}$ for
$$
\begin{aligned}
\theta_1&=1\ot v_3,
\qquad
\, \theta_3=x_1\ot v_1+x_2\ot v_2+x_3\ot v_3+x_4\ot v_4,\\
\theta_2&=1\ot v_4,
\qquad
\, \theta_4=x_1^p\ot v_1+x_2^p\ot v_2+x_3^p\ot v_3+x_4^p\ot v_4,
\\
\omega_1&=1\ot x_1,
\qquad
\omega_3=x_3\ot x_1+x_4\ot x_2-x_1\ot x_3- x_2\ot x_4, \\
\omega_2&=1\ot x_2,
\qquad
\omega_4=x_3^p\ot x_1+x_4^p\ot x_2-x_1^p\ot x_3- x_2^p\ot x_4
\, .
\end{aligned}
$$
None of the transvection root spaces
are maximal, and the conclusion
of \cref{AllButOne} fails:
$$\{d\theta_E\}\cup\{\omega_i\theta_j:(i,j)\neq (a,b)\}
\text{ is {\em not} an }
S^G
\text{-basis of }
(S\ot V^*\ot V)^G
\text{ for any }
(a,b).
$$
}
\end{example}

\begin{remark}{\em
The arguments of \cref{dualizingsection,structuresection}
apply when $\characteristic \FF\neq 2$
to any finite group
$G$ with transvection root spaces 
maximal if $\det(g)^e=1$ for every $g\in G$, where $e$ is the maximal order of a diagonalizable reflection in $G$.
Indeed, in this case, $Q_{\det}$ is $\det$-invariant and the arguments in the proofs
of \cref{dualizingsection,structuresection}
show that $(S\ot \wedge V^*\ot V)^G$ is free over $S^G$ when $(S\ot V)^G$ is free over $S^G$, with explicit 
basis given by \cref{AllButOne}
or \cref{ExplicitBasis}.
However, we have yet to even find an example of a nonreflection group whose transvection root spaces are all maximal with $(S\ot V)^G$ free over $S^G$.
} 
\end{remark}

\section{The Case of Characteristic 2}
\label{char2}

In this section, we consider the case when $\characteristic\FF$ is $2$ and $G$ is a reflection group with transvection root spaces maximal. 
Examples include
$\SL_n(\FF_q)$ and $\GL_n(\FF_q)$
for $q$ a power of $2$.
The structure
of $(S\ot \wedge V^*\ot V)^G$ 
may differ
from that 
in \cref{structuresection} where $\characteristic \FF\neq 2$.
Indeed, 
in \cref{appendix}, we must distinguish the groups whose pointwise stabilizers of hyperplanes consist of exactly one transvection and the identity element, i.e., $\delta_H=1$
for all $H$ in the reflection arrangement $\A$,
which only occurs
when $\characteristic \FF=2$ (see \cref{deltaH}).

\begin{theorem}
\label{char2structure}
Let $G\subset \GL(V)$ be a reflection group with transvection root spaces all maximal and $\characteristic \FF=2$.
\begin{itemize}\setlength{\itemindent}{-3ex}
\item If $|\A|=1$, $(S\ot V)^G$, $(S\ot V^*)^G$, and 
$(S\ot \wedge V^*\ot V)^G$ 
are free $S^G$-modules with structure given by  \cref{BasicDerivationsOneHyperplane} and \cref{OneHyperplane}.
\item If $|\A|\neq 1$
and $G_H$ comprises a single transvection and the identity for each $H\in \A$,
then $(S\ot V)^G$, $(S\ot V^*)^G$, and 
$(S\ot \wedge V^*\ot V)^G$ are free $S^G$-modules.
In this case, $n=2$ and
$(S\ot \wedge V^*\ot V)^G$ has rank $4$
over an $S^G$-submodule
of $(S\ot \Wedge V^*)^G$
of rank $2$.
\item Otherwise, 
$(S\otimes \Wedge V^*\otimes V)^G$ is a free $S^G$-module, provided $(S\ot V)^G$ is a free $S^G$-module, with bases given by
\cref{AllButOne,ExplicitBasis}
and module structure
given by \cref{structuretheorem2}.
\end{itemize}
\end{theorem}

\begin{proof}
For the first claim, we just appeal to
\cref{BasicDerivationsOneHyperplane} and \cref{OneHyperplane}.
In the case of the third claim, 
$\delta_H=0$
for some $H\in \A$ (see
\cref{deltaH}) so in fact $\delta_H=0$ for all $H\in\A$ by \cref{actstransitively}
(see \cref{OneOrbit}). Then the arguments in the proofs of
 \cref{AllButOne,ExplicitBasis}
and \cref{structuretheorem2}
hold.

Now assume we are in the setting of the second
claim.
Then $\delta_H=1$ for every $H\in\A$
and
there are nonnegative integers $e$, $b$, and $a_k$
such that $e=e_H$, $b=b_H$, and $a_k=a_{H,k}$ for all $H\in\A$ (again see \cref{OneOrbit}).
In this case, $G$ contains no diagonalizable reflections and $e=1$. 
Further, each transvection root space of $G$ 
has dimension $b=1$
as it is
spanned by a single transvection.
But this forces $\dim V=n=2$
as each transvection root space is also maximal, and thus
$G$ is a finite subgroup of $\SL_2(\FF)$
as $G$ is generated by
its transvections.
Then $G$ must be
isomorphic (as
an abstract group) to some dihedral group
$D_{2m}$
of order $2m$ 
with $m$ odd
by the classification of Dickson (see \cite[Chapter~3, Section~6]{Suzuki}).
There
are exactly $m$ elements
of order $2$ in $D_{2m}$
as $m$ is odd, 
and thus $|\A|=m$ is odd
as the transvections are the only elements of order $2$ in $G$ and
there is only one transvection per
hyperplane. 
Hence $\deg Q$ is greater than $2$ and is odd as $|\A|\neq 1$.

As $G$ is generated
by two transvections,
there is some
 $\alpha\in\FF^\times$
 and
a basis $v_1, v_2$
of $V$ with dual basis
$x_1, x_2$ of $V^*$
so that 
$
G
=
\left\langle\begin{smallpmatrix}
1&1\\0&1
\end{smallpmatrix},
\begin{smallpmatrix}
1&0\\\alpha &1\end{smallpmatrix}
\right\rangle
$.
The following derivations
$\theta_1,\theta_2$
are an $S^G$-basis of $(S\ot V)^G$ (see, e.g., \cite[Section~B.2]{OrlikTerao}), and their dual $1$-forms $\omega_1,\omega_2$ are an $S^G$-basis of $(S\ot V^*)^G$ by \cref{dualderivations}:
$$
\begin{aligned}
\theta_1&=\tfrac{\partial Q}{\partial x_2}\otimes v_1+\tfrac{\partial Q}{\partial x_1}\otimes v_2
\, ,
&\theta_2&=x_1\otimes v_1+x_2\otimes v_2\, ,\\
\omega_1&=x_2\ot x_1+x_1\ot x_2
\, , 
&\omega_2&=\tfrac{\partial Q}{\partial x_1}\otimes x_1+\tfrac{\partial Q}{\partial x_2}\otimes x_2\, .
\end{aligned}
$$
We now give an explicit $S^G$-basis for each $(S\ot \wedge^k V^*\ot V)^G$
using \cref{saitotheorem}.

For $k=0$, the derivations $\theta_1,\theta_2$ are an $S^G$-basis as they are $\F(S)$-independent and have polynomial degrees summing to $\deg Q$.

For $k=1$, we argue that the four forms $\omega_1\theta_1,\omega_1\theta_2,d\theta_E,\eta_0$ are $\F(S)$-independent,
lie in
$(S\ot V^*\ot V)^G$,
and have
polynomial degrees summing to $2\deg Q$
for
$$
\eta_0=Q^{-1}(\omega_2\theta_1+f^{\deg Q-2}\omega_1\theta_2)
\quad
\text{ with }\ \ 
f=x_1^2+x_1 x_2 + \alpha^{-1}
x_2^2
\, \, \in S^G
\, .
$$
First notice that the forms $\omega_1\theta_1,\omega_1\theta_2,\omega_2\theta_1,\omega_2\theta_2$ are $\F(S)$-independent by \cref{wholesetindependent}, which implies that
$\omega_1\theta_1,\omega_1\theta_2,\omega_2\theta_1,d\theta_E$ are also $\F(S)$-independent since
$$
\omega_1\theta_1 + \omega_2\theta_2
=
(x_1\tfrac{\partial Q}{\partial x_1} + 
x_2\tfrac{\partial Q}{\partial x_2})\otimes x_1\otimes v_1 + 
(x_1\tfrac{\partial Q}{\partial x_1} + 
x_2\tfrac{\partial Q}{\partial x_2})\otimes x_2\otimes v_2
=(\deg Q)\, Q\, d\theta_E \neq 0
$$
as $\deg Q=|\A|$ is odd
and $Q$ is homogeneous, using Euler's identity.
This implies that
$\omega_1\theta_1,\omega_1\theta_2,d\theta_E,\eta_0$ 
are also $\F(S)$-independent.

Second,
we claim that $\eta_0\in
(S\ot V^*\ot V)^G$, i.e., that
$\eta_0$ has polynomial coefficients
and is $G$-invariant.
Since $f$ and $Q$ are $G$-invariant  (as $e=1$), so is $\eta_0$.
Observe that
\begin{small}
$$
\begin{aligned}
Q\eta_0=&\Big(\tfrac{\partial Q}{\partial x_1}\tfrac{\partial Q}{\partial x_2}+f^{\deg Q -2}
x_1x_2\Big)\otimes x_1\otimes v_1
+\Big(\big(\tfrac{\partial Q}{\partial x_2}\big)^2+f^{\deg Q -2}x_1^2\Big)\otimes x_2\otimes v_1
\\
&+\Big( \big(\tfrac{\partial Q}{\partial x_1} \big)^2
+f^{\deg Q -2}x_2^2\Big)\otimes x_1\otimes v_2
+\Big(\tfrac{\partial Q}{\partial x_1}\tfrac{\partial Q}{\partial x_2}
+f^{\deg Q -2}x_1x_2\Big)\otimes x_2\otimes v_2
\, .
\end{aligned}
$$
\end{small}%
Notice that $x_2$ divides $\tfrac{\partial Q}{\partial x_1}$ 
(apply \cite[Lemma~4]{HartmannShepler} to $\om_2$), which
implies that $x_2$
divides the first, third, and last coefficients in this expression.
Then, as the factors of $Q$ are relatively prime, $x_2$ does not divide $\tfrac{\partial Q}{\partial x_2}$,
and we may rescale $Q$
without loss of generality
so that the term
$x_1^{2\deg Q -2}$ in $\big(\tfrac{\partial Q}{\partial x_2}\big)^2$ cancels
with the same term
in $f^{\deg Q -2} x_1^2$,
which implies that $x_2$ divides
the second coefficient
as well.
Hence, $Q\eta_0\in x_2 S\ot V^*\ot V$.
Then as the reflecting hyperplanes of $G$ are 
all in the same orbit
by \cref{actstransitively}
and $Q\eta_0$ is invariant,
$Q\eta_{0}\in 
\ell_H S\ot V^*\ot V$
for any $H$ in $\A$, and thus
$Q\eta_0\in Q S\ot V^*\ot V$. Hence, $\eta_0$ indeed lies in 
$(S\ot V^*\ot V)^G$.
Finally, note that the polynomial
degrees of
$\omega_1\theta_1,\omega_1\theta_2,d\theta_E,\eta_0$ 
add to $2\deg Q$, and thus
these differential derivations are an $S^G$-basis for $k=1$. 

For $k=2$, the forms
$\omega_1 d\theta_E$
and $\omega_1\eta_0$ are an $S^G$-basis since they
are $\F(S)$-independent
with polynomial degrees that add to $\deg Q$.

Hence
$\theta_1, \theta_2, \omega_1\theta_1,\omega_1\theta_2,d\theta_E,\eta_0,\omega_1d\theta_E, \omega_1\eta_0$ are an $S^G$-basis
of
$(S\ot\wedge V^*\ot V)^G$.
Thus
$$
(S\ot \Wedge V^*\ot V)^G
=R\text{-span}\{\theta_1, \theta_2,
d\theta_E, \eta_0\}
$$
for $R=S^G\text{-span}\{1\ot 1, \om_1\}\subset(S\ot \Wedge V^*)^G$
and the $R$-module
$(S\ot \Wedge V^*\ot V)^G$
has rank $4$.

Alternatively, we note that
$\theta_1, \theta_2, \omega_1\theta_1,\omega_1\theta_2,\eta_0,d\theta_E,
(\omega_1\curlywedge \omega_2)\theta_1,(\omega_1\curlywedge\omega_2)\theta_2$ also
are an $S^G$-basis
of
$(S\ot\wedge V^*\ot V)^G$.
This is because
$(\omega_1\curlywedge \omega_2)\theta_1,(\omega_1\curlywedge\omega_2)\theta_2$
are an $S^G$-basis
for $k=2$:
they
are $\F(S)$-independent
with polynomial degrees that add to $\deg Q$.
We compare with
\cref{AllButOne} and
observe that this alternate $S^G$-basis 
of
$(S\ot\wedge V^*\ot V)^G$
is
$$\{d\theta_E,\eta_0\}
\cup
\{\omega_I^{\curlywedgeup}\theta_1,\omega_I^{\curlywedgeup}\theta_2: I\subset [2] \}\setminus\{\omega_2\theta_1,\omega_2\theta_2\}
\, .
\vspace{-5ex}
\hphantom{xx}
$$
\end{proof}

\cref{char2structure}
and \cref{freeimpliesfree}
imply the following.
\begin{cor}
\label{FreeImpliesFree}
Let $G\subset \GL(V)$ be a reflection group with transvection root spaces all maximal
and $\characteristic \FF$ arbitrary.
If $(S\ot V)^G$ is a free $S^G$-module, then so is
$(S\ot \wedge V^*\ot V)^G$.
\end{cor}

\begin{example}
\label{sl2f2example}
{\em 
Let $G=\SL_2(\FF_2)$.
Here, $Q=x_1^2x_2+x_1x_2^2$ and each $G_H$ consists of exactly one transvection and the identity, so $e=b=\delta=1$. 
Then $(S\ot V)^G$ is free over $S^G$ with basis $\theta_1,\theta_2$ and $(S\ot V^*)^G$  is free over $S^G$ with basis $\omega_1,\omega_2$ (see \cref{saitoderivations}) for
$$\begin{aligned}
\theta_1&=x_1^2\otimes v_1+x_2^2\otimes v_2,&
\theta_2=x_1\otimes v_1+x_2\otimes v_2,\\
\omega_1&=x_2 \otimes x_1+x_1\otimes x_2,&
\omega_2=x_2^2\otimes x_1 + x_1^2\otimes x_2.
\end{aligned}
$$
As $x_1^2+x_1x_2+x_2^2$ lies in $S^G$, the proof of \cref{char2structure} gives an $S^G$-basis of $(S\ot \wedge V^*\ot V)^G$:
\begin{small}
$$
\begin{aligned}
\theta_1&=x_1^2\ot 1\ot v_1+x_2^2\ot 1\ot v_2,\quad
\theta_2=x_1\ot 1\ot v_1 + x_2\ot 1\ot v_2,
\\
\omega_1\theta_1&=x_1^2x_2\ot x_1\ot v_1+x_2^3\ot x_1\ot v_2 + x_1^3\ot x_2\ot v_1 + x_1x_2^2\ot x_2\ot v_2,
\\
\omega_1\theta_2&=x_1x_2\ot x_1\ot v_1+x_2^2\ot x_1\ot v_2 + x_1^2\ot x_2\ot v_1 + x_1x_2\ot x_2\ot v_2,
\\
d\theta_E&=1\ot x_1\ot v_1 + 1\ot x_2\ot v_2, 
\\
\eta_0&=(x_1+x_2)\otimes x_1\otimes v_1+x_2\otimes x_1\otimes v_2+x_1\otimes x_2\otimes v_1+(x_1+x_2)\otimes x_2\otimes v_2,
\\
\omega_1d\theta_E&=x_1\ot x_1\wedge x_2\ot v_1+x_2\ot x_1\wedge x_2\ot v_2
\, ,
\quad
\omega_1\eta_0=x_1^2\ot x_1\wedge x_2\ot v_1+x_2^2\ot x_1\wedge x_2\ot v_2
\, .
\end{aligned}
$$
\end{small}
}
\end{example}

\section{Prime Fields}
\label{fpsection}
We now consider finite groups $G$ acting on vector spaces over a prime field $\FF=\FF_p$ for a fixed prime $p$. 
We observe that
$(S\ot V)^G$, $(S\ot \Wedge V^*)^G$ and $(S\ot \wedge V^*\ot V)^G$ are free $S^G$-modules
when $G$ is a reflection group
with transvection root spaces maximal
and produce bases.
Examples include $\SL_n(\FF_p)$ and $\GL_n(\FF_p)$, see the next section.

\subsection*{Reflection arrangements}
We first examine
arrangements over prime fields when the
transvection roots spaces are maximal.
Recall that $\A=\A(G)$ is the collection of reflecting hyperplanes of a group $G$
acting linearly.
\begin{lemma}
\label{extrahyperplanes}
Let $G\subset \GL(V)$ be a finite group acting on $V=\FF_p^n$, for $p$ a prime.
If $H, H'\in \A$ with 
the transvection root space of $H$ maximal, then 
$\A$ contains 
all hyperplanes in $V$ containing $H\cap H'$.
\end{lemma}

\begin{proof}
Say $H=\ker \ell_H$ and $H'=\ker \ell_{H'}$ for $\ell_H,\ell_{H'}\in V^*$.
A hyperplane in $V$ containing
$H\cap H'$ must be the kernel of
$\ell_{H'}+ c \, \ell_H$ 
for some $c \in \FF_p$.
As the transvection root space of $H$ is maximal, $G$ contains a transvection $t$ 
about $H$
whose root vector $v_t$ lies outside of $H'$, i.e., $\ell_{H'}(v_t)\neq 0$.
Set $a= (\ell_{H'}(v_t))^{-1} c$ in $\FF_p$ and 
regard $a$ as an integer.
A straightforward calculation confirms that the kernel of $\ell_{H'}+c\ell_H$ 
is the reflecting hyperplane of 
$t^{-a}\, s'\, t^a$
for any reflection $s'$ in $G$ about $H'$
and thus lies in $\A$.
 \end{proof}

\begin{proposition}
\label{nicearrangement}
Let $G\subset \GL(V)$ be a finite group 
acting on $V=\FF_p^n$,
for $p$ a prime,
with transvection 
root spaces all maximal. 
Then the reflection arrangement of $G$ coincides with that for the general linear group 
(embedded in $\GL(V)$)
of some subspace $W$ of $V$: 
$$
\A(G)=\A(\GL(W)) 
\quad\text{ for some }
\GL(W)\subset \GL(V)\, .
$$
Thus there is a basis $x_1,\dots,x_n$ of $V^*$ 
with $\A(G)$
defined by, for some $m$, 
\begin{small}
$$
Q=x_1
\Big(\prod_{\alpha_1\in\FF_p}
\hspace{-1ex}
x_2+\alpha_1x_1\Big)
\Big(\prod_{\alpha_1,\alpha_2\in\FF_p}\hspace{-2ex}
x_3+\alpha_2x_2+\alpha_1x_1\Big)\cdots \Big(\prod_{\alpha_1,\dots,\alpha_{m-1}\in\FF_p}
\hspace{-3ex}
x_m+\alpha_{m-1}x_{m-1}+\dots+\alpha_1x_1\Big)
\, .
$$
\end{small}
\end{proposition}

\begin{proof}
Let $H_1\in\mathcal{A}=\A(G)$ be arbitrary and set $\mathcal{A}_1=\{H_1\}$.
Inductively choose some $H_i\in\mathcal{A}\backslash\mathcal{A}_{i-1}$ and set $\mathcal{A}_i=\{ H\in\mathcal{A} : H\supset H_1\cap\cdots\cap H_i\}$ to obtain a maximum set
of hyperplanes $H_1,\dots,H_m$ 
for which  $\mathcal{A}_m=\mathcal{A}$.
Choose $x_i$ in $V^*$ so that
$H_i=\ker(x_i)\in\A$
and notice that $x_1, \ldots, x_m$
is $\FF$-independent since
$\dim ( H_1\cap \ldots \cap H_i )
=n-i$ for all $i\leq m$.
We extend to a basis
$x_1, \ldots, x_n$ of $V^*$.
By \cref{extrahyperplanes},
any nonzero linear combination of 
linear forms defining hyperplanes in $\A$
defines a hyperplane again in $\A$.
Thus for each $i\leq m$,  
$\A_i=\{H : \ell_H\in\FF_p\text{-span}\{x_1,\dots,x_i\}\}$,
and the claim follows.
\end{proof}

\subsection*{Free arrangements}
Recall that an arrangement
of hyperplanes $\A$ is {\em free} if the set of derivations $D(\A)$ along the arrangement
is a free $S$-module,
see \cite{OrlikTerao}, where
$$
D(\A)=
\{\theta\in\text{Der}_S : \theta(\ell_H)\in \ell_HS\text{ for all }H\in\A\}\, . 
$$ 
(Recall that we 
identify
$\sum_i f_i \ot v_i$
in $S\ot V$
with the derivation
$\sum_i f_i\ot \del/\del x_i$.)
Bases for the free modules
in the next corollary 
are given
in \cref{dualderivations} and \cref{AllButOne} using the derivations in the proof.
Also see the proofs of
\cref{char2structure}
and \cref{OneHyperplane}.

\begin{cor}
If $G \subset GL(V )$
is a finite group acting on 
$V = \FF_p^n$ for $p$ a prime
with transvection root spaces all maximal,
then
$\A(G)$ is a free arrangement.
If, in addition,
$G$ is a reflection group,
then  $(S\ot V)^G$,
$(S\ot \wedge V^*)^G$,
and $(S\ot \wedge V^*\ot V)^G$ 
are free $S^G$-modules. 
\end{cor}
\begin{proof}
By \cref{nicearrangement},
 $\A=\A(G)=\A(\GL(W))$
for a subspace $W$ of $V$
of dimension $m$.
We use the basis $x_1,\ldots, x_n$ of $V^*$ of \cref{nicearrangement}
and dual basis $v_1, \ldots, v_n$ of $V$ and
set
$$
\theta_i=
\left\{
\begin{aligned}
&
\sum_{j=1}^n x_j^{p^{m-i}}\otimes v_j
&&\quad\text{for }1\leq i\leq m\, ,\\
&1\otimes v_i
&&\quad\text{for }m< i\leq n
\end{aligned}\right.
$$
so that $\det\coef(\theta_1, \ldots,
\theta_n)\doteq Q$.
Then as each $\theta_i$ lies in $D(\A)$,
the $\theta_i$ generate
$D(\A)$ as an $S$-module
and $\A$ is a free arrangement
by
the original Saito's Criterion
\cite[Theorem 4.19]{OrlikTerao}.
Now assume further that $G$
is a reflection group.
Notice that each $\theta_i$ for $i \leq m$
is invariant under $\GL_n(\FF_p)$
(see \cref{slglsection})
and that $\theta_i$
for $m< i$
is invariant under each reflection of $G$
since $v_{m+1},\ldots, v_n$
lie in $\bigcap_{H\in\A} H$.
Hence 
$\theta_1, \ldots, \theta_n$ are $G$-invariant
and
are an $S^G$-basis of $(S\ot V)^G$
by \cref{saitoderivations}.
Then $(S\ot \wedge V^*)^G$ 
and 
$(S\ot \wedge V^*\ot V)^G$
are both
free $S^G$-modules
by \cref{DualityOfExponents}
and
\cref{FreeImpliesFree}.
\end{proof}

\section{Special and General Linear Groups and Groups in between}
\label{slglsection}
We now turn our attention to
the special linear group, the general linear group,
and all groups in between over a finite field
$\FF_q$ for $q$ a prime power.
Let $G$ be a group with  
$\SL_n(\FF_q)\subset G\subset \GL_n(\FF_q)$. 
Then $G$ is generated
by reflections,
and, as $G$ contains $\SL_n(\FF_q)$,
each transvection root space for $G$ is maximal
and
there is a single
orbit of reflecting hyperplanes
(see \cref{actstransitively}).
The maximal order of a diagonalizable reflection in $G$ is $e:=|G:\SL_n(\FF_q)|$.
Here, $\A=\A(G)$ is the collection
of all hyperplanes $H$ in $V=\FF_q^n$
and its defining polynomial $Q=\prod_{H\in \A} \ell_H$
thus has degree
$|\A|=[n]_q=1+q+\dots+q^{n-1}$.

\subsection*{Invariant polynomials}
Basic invariant
polynomials 
$f_1,\ldots, f_n$
with 
$S^G=\FF[f_1,\ldots, f_n]$
are given in terms
of the classical
Dickson invariants
$D_{n,i}$ (see \cite{Steinberg}
and \cite{SmithBook})
with $\deg D_{n,i}= q^n-q^i$
for $i=0, \ldots, n-1$:
$$
f_1=Q^e
\ \ 
\text{and}
\ \ 
f_i=D_{n,i-1}
\ \ \text{for } 2\leq i\leq n
\, .
$$

\subsection*{Invariant derivations}
Here,
$(S\ot V)^G$ is a free $S^G$-module 
with basis
$$
\theta_i
=\sum_{j=1}^n
x_j^{q^{n-i}}\otimes v_j
\ \ 
\text{for } 1\leq i\leq n
\, 
$$
with respect to a fixed ordered
basis $v_1, \ldots, v_n$
of $V$ and dual basis
$x_1, \ldots, x_n$ of $V^*$ 
(see \cite[Example~4.24]{OrlikTerao}) 
since
$Q=\det\coef(\theta_1,\ldots,\theta_n)$
after rescaling $Q$ if necessary 
(see \cref{saitoderivations}).

\subsection*{Invariant $\mathbf{1}$-forms}
\cref{dualderivations}
gives a dual $S^G$-basis 
of $(S\ot V^*)^G$:  explicitly,
let $\om_1,\dots,\om_n$ in $S\ot V^*$ be the $1$-forms
whose coefficient matrix is
(for $t$ indicating transpose)
$$\coef(\om_1,\ldots,\om_n)
=
Q^e\,\big( \coef(\theta_1,\ldots,\theta_n)\big)^{-t}
\, .
$$
Then $\det\coef(\om_1,\ldots,\om_n)=Q^{en-1}$
and $\om_1,\dots,\om_n$ are a free
$S^G$-basis of $(S\ot V^*)^G$ by 
\cref{saitoderivations}.  These moreover
generate $(S\ot \wedge V^*)^G$
via the twisted wedging of
\cref{twistedwedgeproduct}:
$
(S\otimes \wedge V^*)^G
=
\bigcurlywedge_{S^G}
\{ \om_1,\ldots, \om_n\}
$
(see \cref{twistedalgebra}
and \cite{HartmannShepler}).
See \cite[Section~6.2]{HartmannShepler} 
for basic $1$-forms
in terms of the exterior derivatives
$df_i$ of the Dickson invariants.


\subsection*{Numerology}
For $m_i=\deg \om_i=e[n]_q-q^{n-i}$
and $m_i^*=\deg \theta_i=q^{n-i}$ (see \cref{DualityOfExponents}),
$$
m_i+m_i^*=
e|\A|
\, .$$ 
Explicitly,
the duality gives
 (also see
\cref{q-Catalan}
and \cref{CoxeterNumber})
$$
\begin{aligned}
&m_i=[n]_q-q^{n-i}
&&\quad\text{ and }\quad{}
m_i^*=q^{n-i}
\quad
&&\text{for $G=\SL_n(\FF_q)\, ,$
and }
\\
&m_i=(q-1)[n]_q-q^{n-i}
&&\quad\text{ and }\quad
m_i^*=q^{n-i}
&&\text{for $G=\GL_n(\FF_q)\, .$}
\end{aligned}
$$

\subsection*{Invariant
differential derivations}
For $G=\SL_2(\FF_2)$, 
we construct in \cref{sl2f2example} 
an explicit basis for 
$(S\ot \wedge V^*\ot V)^G$ as a free $S^G$-module.
\cref{AllButOne} and \cref{char2structure} 
imply a similar result for all other groups 
between $\SL_n(\FF_q)$ and $\GL_n(\FF_q)$:
\begin{cor}
Let $G$ be a group with $\SL_n(\FF_q)\subset G \subset \GL_n(\FF_q)$
and $G\neq \SL_2(\FF_2)$.
Then
$(S\otimes \wedge V^*\otimes V)^G$ is a free $S^G$-module with basis
$$
\{ d\theta_E\}
\cup
\big\{
(\om_{i_1}\curlywedge\cdots\curlywedge\om_{i_k})\, \theta_j :
1\leq i_1<\ldots < i_k\leq n,
1\leq j\leq n,
0\leq k\leq n\}
\setminus \{\om_n \theta_n \}
\, .
$$
\end{cor}

\begin{example}{\em
For the reflection group $G=\SL_2(\FF_3)$
acting on $V=\FF_3^2$, 
$$
\begin{aligned}
&\text{ basic derivations }
&&\theta_1=x_1^3\ot v_1+x_2^3\ot v_2\, ,
&&\theta_2 = x_1\ot v_1 + x_2\ot v_2
\  \text{ and }
\\
&\text{ basic $1$-forms } 
&&\omega_1=x_2\ot x_1- x_1\ot x_2\, ,
&&\omega_2=-x_2^3\ot x_1 + x_1^3\ot x_2
\end{aligned}
$$
generate $(S\ot V)^G$ and $(S\ot V^*)^G$, respectively, as free $S^G$-modules.
Then the $S^G$-module $(S\otimes \wedge V^*\otimes V)^G$
is also free with basis
\begin{small}
$$
\begin{aligned}
\theta_1&=x_1^3\ot 1\ot v_1 + x_2^3\ot 1\ot v_2\, ,\quad
\theta_2=x_1\ot 1\ot v_1 + x_2\ot 1\ot v_2\, ,\\
\omega_1\theta_1&=x_1^3x_2\ot x_1\ot v_1 + x_2^4\ot x_1\ot v_2 - x_1^4\ot x_2\ot v_1 - x_1x_2^3\ot x_2\ot v_2\, , \\
\omega_1\theta_2&=x_1x_2\ot x_1\ot v_1 + x_2^2\ot x_1\ot v_2 - x_1^2\ot x_2\ot v_1 - x_1x_2\ot x_2\ot v_2\, ,\\
\omega_2\theta_1&=-x_1^3x_2^3\ot x_1\ot v_1 - x_2^6\ot x_1\ot v_2 + x_1^6\ot x_2\ot v_1 + x_1^3x_2^3\ot x_2\ot v_2\, ,\\
d\theta_E&=1\ot x_1\ot v_1+1\ot x_2\ot v_2\, ,\\
(\omega_1\curlywedge \omega_2)\theta_1&=x_1^3\ot x_1\wedge x_2\ot v_1 + x_2^3\ot x_1\wedge x_2\ot v_2\, ,\\
(\omega_1\curlywedge\omega_2)\theta_2&=x_1\ot x_1\wedge x_2\ot v_1 + x_2\ot x_1\wedge x_2\ot v_2\, .
\end{aligned}
$$
\end{small}
Here, $S^G=\FF_3[f_1, f_2]$ for
$f_1=x_1^3x_2-x_1x_2^3$ and $f_2=x_1^6+x_1^4x_2^2+x_1^2x_2^4+x_2^6$.
}
\end{example}

\begin{example}{\em 
For the reflection group $G=\GL_2(\FF_3)$
acting on $V=\FF_3^2$,
the $S^G$-modules $(S\ot V)^G$
and $(S\ot V^*)^G$ are both free with respective
bases 
$\theta_1, \theta_2$
(basic derivations)
and 
$\om_1, \om_2$
(basic $1$-forms)
given by
\begin{small}
$$
\begin{aligned}
&\theta_1=x_1^3\otimes v_1+x_2^3\otimes v_2\, ,\quad \theta_2=x_1\otimes v_1+x_2\otimes v_2\, ,\\
&\omega_1=(x_1^3x_2^2-x_1x_2^4)\otimes x_1 
+ (x_1^2x_2^3-x_1^4x_2)\otimes x_2,\ 
\omega_2=(x_1x_2^6-x_1^3x_2^4)\otimes x_1 + (x_1^6x_2-x_1^4x_2^3)\otimes x_2\, .
\end{aligned}
$$
\end{small}
The $S^G$-module $(S\otimes \wedge V^*\otimes V)^G$
is then free with
basis
\begin{small}
$$
\begin{aligned}
\theta_1=&\ x_1^3\ot 1\ot v_1 + x_2^3\ot 1\ot v_2\, ,\quad
\theta_2=x_1\ot 1\ot v_1 + x_2\ot 1\ot v_2\, ,\\
\omega_1\theta_1
=&\ (x_1^6x_2^2-x_1^4x_2^4)\ot x_1\ot v_1+(x_1^3x_2^5-x_1x_2^7)\ot
x_1\ot v_2
\\
&
+(x_1^5x_2^3-x_1^7x_2)\ot x_2\ot v_1+(x_1^2x_2^6-x_1^4x_2^4)\ot x_2\ot v_2\, ,\\
\omega_1\theta_2
=& \ (x_1^4x_2^2-x_1^2x_2^4)\ot x_1\ot v_1+(x_1^3x_2^3-x_1x_2^5)\ot x_1\ot v_2
\\
&
+(x_1^3x_2^3-x_1^5x_2)\ot x_2\ot v_1+(x_1^2x_2^4-x_1^4x_2^2)\ot x_2\ot v_2\, ,\\
\omega_2\theta_1
=&\ (x_1^4x_2^6-x_1^6x_2^4)\ot x_1\ot v_1+(x_1x_2^9-x_1^3x_2^7)\ot x_1\ot v_2
\\
&
+(x_1^9x_2-x_1^7x_2^3)\ot x_2\ot v_1+(x_1^6x_2^4-x_1^4x_2^6)\ot x_2\ot v_2,\\
d\theta_E =&\ 1\ot x_1\ot v_1+1\ot x_2\ot v_2\, ,\\
(\omega_1\curlywedge\omega_2)\theta_1=
&\ (x_1^6x_2-x_1^4x_2^3)\ot x_1\wedge x_2\ot v_1 + (x_1^3x_2^4-x_1x_2^6)\ot x_1\wedge x_2\ot v_2
\, , \\
(\omega_1\curlywedge\omega_2)\theta_2=
&\ (x_1^4x_2-x_1^2x_2^3)\ot x_1\wedge x_2\ot v_1 + (x_1^3x_2^2-x_1x_2^4) \ot x_1\wedge x_2\ot v_2
\, .
\end{aligned}
$$
\end{small}
Here, $S^G=\FF_3[f_1, f_2]$ for
 $f_1=x_1^6+x_1^4x_2^2+x_1^2x_2^4+x_2^6$ and $f_2=x_1^6x_2^2+x_1^4x_2^4+x_1^2x_2^6$.
}
\end{example}

\appendix
\section{}
\label{appendix}
The technical analysis in this appendix
provides the heavy lifting for determining the Saito criterion for invariant differential derivations in \cref{saitocriterion}.
Throughout this section, we fix 
a nontrivial 
finite group $G\subset \GL(V)$ acting on 
$V=\FF^n$ that
fixes a single hyperplane $H=\ker \ell$ in $V$
for some linear form $\ell$ in $V^*$.
We fix $e=e_H\geq 1$ and $b=b_H\geq 0$ throughout 
and use the basis 
$v_1,\dots,v_n$ of $V$ with 
dual basis $x_1,\dots,x_n$ of $V^*$
as in 
\cref{ConvenientBasis}, as well as
the (possibly empty) set of transvections $t_1,\dots,t_{b}$ in $G$
and an element
$s$ in $G$
with either $s=1_G$
when $e=1$
or $s$ is a 
diagonalizable reflection of maximal order $e>1$ in $G$. 

\subsection*{Action on basis elements.}
We record the action of $s$ and each
transvection $t_m$ 
on basis elements $v_i$ of $V$ and 
$x_I$ of $\wedge V^*$.
For a
fixed $m$ and $I=\{i_1, \ldots, i_k\}\subset\{1,\dots,n\}$
with $i_1<\ldots< i_k$,
$n\notin I$, and $m\in I$,
define $\ep_{I,m}=\pm 1$ 
by
$$
x_{\sigma(I)}=\ep_{I,m} \ x_{\sigma(i_1)} \wedge \cdots \wedge x_{\sigma(i_k)}
\, 
\quad\text{ for the transposition }
\sigma=(m\ n)
\, .
$$
Also, set $\lambda=\det(s)$. 
Then for $1\leq m\leq b$,
\begin{small}
 $$\begin{aligned}
t_m(x_I)=&
\begin{cases}x_I& \text{when $n\in I$ or $m\notin I$},
\\
x_I - \ep_{I,m} x_{\sigma(I)} & \text{when $n\notin I$ and $m\in I$},
\end{cases}
&&
&t_m(v_j)=&
\begin{cases}
v_j&\;\,
\text{when $j\neq n$}
\, ,
\\
v_m+v_n&\;\,
\text{when $j=n$}
\, ,
\end{cases}
\\
s(x_I)=&
\begin{cases}
x_I&
\text{when $n\notin I$},\\
\lambda^{-1}x_I
\hphantom{xxxxxx}
& 
\text{when $n\in I$},
\end{cases}
\qquad
&&
&\text{and}\ \ s(v_j)=&
\begin{cases}
v_j&
\quad\text{when $j\neq n$}
\, ,\\
\lambda v_n
\hphantom{xx}&
\quad\text{when $j=n$}
\, .
\end{cases}
\end{aligned}$$
\end{small}
Note that
$\ep_{I,m}=1$ and $t_m(x_I)=x_m-x_n$
when $I=\{m\}$.

\subsection*{Action on polynomials}
We require some straightforward observations.
\begin{lemma}
\label{sf-f}
For any reflection $g$ 
about 
$H=\ker \ell$
and any polynomial $f$ in $S$,
$\ell$ divides $g(f)-f$. Also,
$\ell^2$ divides $g(f)-f$
whenever $\ell$ divides $f$.
\end{lemma}

\begin{lemma}
\label{modeh}
Let $\det(s)=\lambda$ of order
$e\geq 1$.
Then for any polynomial $f$,
\begin{enumerate}[label=(\alph*)]
\item $s(f)=\lambda f$ implies $\ell^{\ts e-1}$ divides $f$,

\item $s(f)=f$ and $\ell$ divides $f$ implies $\ell^{\ts e}$ divides $f$,

\item $s(f)=\lambda^{-1}f$ and $\lambda\neq 1$ implies $\ell$ divides $f$, and

\item $s(f)=\lambda^{-1}f$ and $\ell^2$ divides $f$ implies $\ell^{\ts e+1}$ divides $f$.

\end{enumerate}
\end{lemma}
\begin{proof}
We prove part (a); the rest follows from 
similar arguments.
Since $s(x_n)=\lambda^{-1}x_n$ and $s$ fixes $x_1,\dots,x_{n-1}$ as well as $x_n^{e}$, 
the degree in $x_n$ of each monomial appearing in $f$ must be $-1\mod e$, and thus
$\ell^{e-1}=x_n^{e-1}$
divides each monomial.
\end{proof}
 
\begin{lemma}
\label{Mxm/xn}
Say $1\leq m\leq b$.
If a monomial $M=x_1^{a_1}\cdots x_{n-1}^{a_{n-1}}x_n$ appears in $t_m(f)-f$ with nonzero coefficient $c$, 
then $a_m+1\neq 0$ in $\FF$ 
and $Mx_m/x_n$ appears in $f$ with nonzero coefficient $-(a_m+1)^{-1}c$.
\end{lemma}
\begin{proof}
For any monomial $x_1^{a_1}\cdots x_n^{a_n}$,
$$
\begin{aligned}
t_m(x_1^{a_1}\cdots x_n^{a_n}) -x_1^{a_1}\cdots x_n^{a_n}
=-a_mx_1^{a_1}\cdots
x_{m-1}^{a_{m-1}}
x_m^{a_m-1}
x_{m+1}^{a_{m+1}}
\cdots x_n^{a_n+1}\; +\; \text{other terms}
.
\end{aligned}$$
The claim then follows from the observation that this expression is zero for $a_m=0$,
and 
otherwise all monomials appearing in this expression have degree in $x_n$ strictly greater than $a_n$, degree in $x_m$ strictly less than $a_m$, and unchanged degrees in the other variables.
\end{proof}

\subsection*{Main lemma} 

At last, the following lemma analyzes the polynomial coefficients of an invariant differential derivation.
Recall that $[b]=\{1, \ldots, b\}$.

\begin{lemma}
\label{lhdivides}
For any differential derivation $\displaystyle \eta=\hspace{-1ex}\sum_{I
\in\tbinom{[n]}{k},\, 1\leq j\leq n}
\hspace{-3ex}
f_{I,j}\otimes x_I\otimes v_j$ in $\displaystyle(S\otimes \wedge^k V^*\otimes V)^{G}$,
\begin{enumerate}

\item [a)] $\ell$ divides $f_{I,n}$ when $n\notin I$,

\item [b)] $\ell^{\ts e}$ divides $f_{I,n}$ when $n\in I$ and $I\cap[b]\neq\varnothing$,

\item [c)] $\ell^{\ts e-1}$ divides $f_{I,j}$ for $j<n$ when $n\in I$,

\item [d)] $\ell^{\ts e}$ divides $f_{I,j}$ for $j<n$ when $n\notin I$ and $I\cap[b]\neq\varnothing$ and $I\cap[b]\neq\{j\}$,

\item [e)] $\ell^{\ts e}$ divides $f_{I,j} - \ep_{I,m} f_{\sigma(I),n}$ for $j=m\leq b$ when $n\notin I$ and $I\cap[b]=\{m\}$
, and

\item [f)] $\ell^{\ts e+1}$ divides $f_{I,n}$ when $n\notin I$ and $I\cap[b]\neq\varnothing$, unless $G$ consists of exactly one transvection and the identity element.
	
\end{enumerate}
\end{lemma}

\begin{proof}
We take all sums over subsets $I \in\tbinom{[n]}{k}$
and $1\leq j\leq n$
as indicated.

\noindent\textbf{Action of transvections.}
Consider the transvection $t_m$ for $m\in[b]$ when $b> 0$
and set $\sigma=(m\ n)$
and $\ep_I=\ep_{I,m}$:
\begin{small}
$$\begin{aligned}
t_m(\eta)
&=
\sum_{I,j}t_m(f_{I,j})\otimes t_m(x_I)\otimes t_m(v_j)\\
&=
\sum_{\substack{I,j :\\n\in I\text{ or }m\notin I,\\j\neq n}}
t_m(f_{I,j})\otimes x_I\otimes v_j\ 
+\sum_{\substack{I :\\n\in I\text{ or }m\notin I}}
t_m(f_{I,n})\otimes x_I\otimes (v_m+v_n)\\
&\;\;\;\;\;
+\sum_{\substack{I,j :\\n\notin I,\, m\in I,\\j\neq n}}
t_m(f_{I,j})\otimes (x_I-\ep_Ix_{\sigma(I)})\otimes v_j\ 
+\sum_{\substack{I :\\n\notin I,\, m\in I}}
t_m(f_{I,n})\otimes (x_I-\ep_Ix_{\sigma(I)})\otimes (v_m+v_n)
\, .
\end{aligned}$$
\end{small}
We reindex and regroup to
express $t_m(\eta)$ as
\begin{small}
$$\begin{aligned}
&\sum_{\substack{I,j :\\n\notin I\text{ or }m\in I,\\j\neq m}}
t_m(f_{I,j}) \otimes x_I\otimes v_j
\ 
+
\sum_{\substack{I :\\n\notin I\text{ or }m\in I}}
\big(t_m(f_{I,m})+t_m(f_{I,n})\big)
\otimes x_I\otimes v_m\\
&\;\;\;\;\;
+\sum_{\substack{I,j :\\n\in I,\, m\notin I,\\j\neq m}}
\big(t_m(f_{I,j})-\ep_{\sigma(I)}t_m(f_{\sigma(I),j})\big)
\otimes x_I\otimes v_j\\
&\;\;\;\;\;
+\sum_{\substack{I :\\n\in I,\, m\notin I}}
\big(t_m(f_{I,m})+t_m(f_{I,n})-\ep_{\sigma(I)}t_m(f_{\sigma(I),m})-\ep_{\sigma(I)} t_m(f_{\sigma(I),n})\big)
\otimes x_I\otimes v_m
\, .
\end{aligned}$$
\end{small}
We equate the polynomial coefficients of $\eta$
and $t_m(\eta)$ and deduce that
\begin{small}
$$
f_{I,j}=
\begin{cases}
t_m(f_{I,j})
&\hspace{-1ex}
\text{for $j\neq m$ when $n\notin I$\;\;\, or $m\in I$},
\\
t_m(f_{I,m})+t_m(f_{I,n})
&\hspace{-1ex}
\text{for $j=m$ when $n\notin I$\;\;\, or $m\in I$},
\\
t_m(f_{I,j})-\ep_{\sigma(I)}t_m(f_{\sigma(I),j})
&\hspace{-1ex}
\text{for $j\neq m$ when $n\in I$ and $m\notin I$},
\\
t_m(f_{I,m}) +t_m(f_{I,n}) - \ep_{\sigma(I)} t_m(f_{\sigma(I),m}) - \ep_{\sigma(I)}t_m(f_{\sigma(I),n})
&\hspace{-1ex}
\text{for $j=m$ when $n\in I$ and $m\notin I$}.
\end{cases}
$$
\end{small}
We solve for $t_m(f_{I,j})$ one case at a time
and conclude
that
\begin{small}
\begin{equation}
\label{transvectionaction}
t_m(f_{I,j})=
\begin{cases}
f_{I,j}
&\text{for $j\neq m$ when $n\notin I$\;\;\, or $m\in I$},
\\
f_{I,m}-f_{I,n}
&\text{for $j=m$ when $n\notin I$\;\;\, or $m\in I$},
\\
f_{I,j}+\ep_{\sigma(I)} f_{\sigma(I),j}
&\text{for $j\neq m$ when $n\in I$ and $m\notin I$},
\\
f_{I,m} - f_{I,n} + \ep_{\sigma(I)} f_{\sigma(I),m} - \ep_{\sigma(I)} f_{\sigma(I),n}
&\text{for $j=m$ when $n\in I$ and $m\notin I$}.
\end{cases}
\end{equation}\\
\end{small}
\noindent\textbf{Action of diagonalizable reflection.}
Since $s$ is diagonal with $\det (s)=\lambda$
of order $e\geq 1$,
\begin{small}
$$\begin{aligned}
s(\eta)&=
\sum_{I,j} s(f_{I,j})\otimes s(x_I)\otimes s(v_j)
=
\sum_{\substack{I,j :n\notin I\\j\neq n}}
s(f_{I,j})\otimes x_I\otimes v_j\ 
+\sum_{I :n\notin I}
\lambda s(f_{I,n})\otimes x_I\otimes v_n\\
&\hspace{30ex}
+\sum_{\substack{I,j : n\in I\\j\neq n}}
\lambda^{-1}s(f_{I,j})\otimes x_I\otimes v_j\ 
+\sum_{I : n\in I}s(f_{I,n})\otimes x_I\otimes v_n.
\end{aligned}
$$
\end{small}
We  equate the polynomial coefficients of $\eta$
and $s(\eta)$ to see that
\begin{small}
\begin{equation}
\label{diagonalizableaction}
s(f_{I,j})=
\begin{cases}
f_{I,j}
&\text{for $j\neq n$ when $n\notin I$},
\\
\lambda^{-1}f_{I,n}
&\text{for $j=n$ when $n\notin I$},
\\
\lambda f_{I,j}
&\text{for $j\neq n$ when $n\in I$},
\\
f_{I,n}
&\text{for $j=n$ when $n\in I$}.
\end{cases}
\end{equation}
\end{small}
Now we use equations \cref{transvectionaction}
and \cref{diagonalizableaction} to show $\ell$ to certain powers divides various $f_{I,j}$
using the fact that $G$
contains either a diagonalizable reflection
or a transvection.\\

\noindent{\bf Parts a) through e).}
For a),
fix $I$ with $n\notin I$.
As $G$ is nontrivial,
either
$G$ contains a transvection $t_m$
or $s\neq 1_G$.
If $G$ contains $t_m$,
then
$f_{I,n}=f_{I,m}-t_m(f_{I,m})$
by \cref{transvectionaction}
so is divisible by $\ell$ by \cref{sf-f}.
If
$s\neq 1_G$, then
$s(f_{I,n})=\lambda^{-1}f_{I,n}$
for $\lambda\neq 1$ by \cref{diagonalizableaction}, so $\ell$ divides $f_{I,n}$ by \cref{modeh}(c).
Either way, $\ell$ divides $f_{I,n}$.
The proof of parts b), c), and
d) are similar.
For part e),
fix $j=m\leq b$ and $I$ with $n\notin I$ and $I\cap[b]=\{m\}$. 
Then
for $\sigma=(m\ n)$
and $\ep_I=\ep_{I,m}$,
$$
t_m(f_{\sigma(I),j})=f_{\sigma(I),j} - f_{\sigma(I),n} + \ep_I f_{I,j} - \ep_I f_{I,n}
$$ 
by \cref{transvectionaction}. Since
$\ell$ divides 
$t_m(f_{\sigma(I),j})-f_{\sigma(I),j}$ by \cref{sf-f} 
and also $f_{I,n}$ by part a), 
it must divide $-f_{\sigma(I),n} + \ep_I f_{I,j}$ 
and hence also
$f_{I,j} - \ep_I f_{\sigma(I),n}$.
Further, 
$\ell^{\ts e}$ divides
$f_{I,j} - \ep_I f_{\sigma(I),n}$ 
by \cref{modeh}(b) 
since it is fixed by $s$ 
(see \cref{diagonalizableaction}) and part e) follows.
\\

\noindent{\bf Part f).}
Complications arise when $\characteristic \FF=2$.
For part f),
assume $G$ does not consist
of exactly one transvection
and the identity,
and fix $I$ with $n\notin I$ and $I\cap[b]\neq\varnothing$. Then $G$ contains a transvection $t_m$ for some $m\in I\cap[b]$ (so $b\neq 0$)
and
\begin{enumerate}

\item [1)] $\characteristic \FF\neq 2$, or

\item [2)] $e>1$, or

\item [3)] $b>1$, or

\item [4)] $\characteristic \FF=2$, $e=1$, and $b=1$, but $G$ contains multiple transvections.
\end{enumerate}
\noindent 
In each case, we will show that $\ell^2$ divides $f_{I,n}$.
Then as  $s(f_{I,n})=\lambda^{-1}f_{I,n}$
by \cref{diagonalizableaction},
\cref{modeh}(d)
will imply that
$\ell^{\ts e+1}$ divides $f_{I,n}$
and the claim 
for part f) will follow.
We fix $m\in I\cap [b]$, $\sigma=(m\ n)$,
and $\ep_I=\ep_{I,m}$.
\\

\noindent \textbf{Case 1: $\mathbf{char\, \FF \neq 2}$.}
Suppose  that $\ell^2$ does not divide $f_{I,n}$. Notice $\ell$ divides $f_{I,n}$ by part a)
so
some monomial $M$ of degree $1$ in $x_n$ appears in $f_{I,n}$
with nonzero coefficient $c\in\FF$. 
As $t_m(f_{I,n})=f_{I,n}$ by \cref{transvectionaction}, 
$$f_{I,n}\in
S^{G'}=\FF[x_m^p-x_mx_n^{p-1},x_i : i\neq m]
\quad\text{ for }
G'=\langle t_m \rangle
\quad \ \ 
$$
(e.g., see \cite{SmithBook}),
and 
the degree of $M$ in $x_m$ is a multiple of $p$ (as $x_n$ divides $f_{I,n}$).
By \cref{transvectionaction},  
$$\begin{aligned}
t_m(f_{I,m}) \;\,&= f_{I,m} - f_{I,n}
\quad\text{and}\quad
t_m(f_{\sigma(I),n}) &= f_{\sigma(I),n} + \ep_I f_{I,n}
\, ,
\end{aligned}
$$ 
so 
$Mx_m/x_n$ appears in $f_{I,m}$ and $f_{\sigma(I),n}$
with nonzero coefficients $c$ and $-\ep_I c$, respectively, by \cref{Mxm/xn}. 
Thus $\ell=x_n$
does not divide
$-f_{\sigma(I),n} + \ep_I f_{I,m}$
(as $\characteristic \FF\neq 2)$.
But
$$t_m(f_{\sigma(I),m}) = f_{\sigma(I),m} - f_{\sigma(I),n} + \ep_I f_{I,m} - \ep_I f_{I,n}
\, ,$$ 
and $\ell$ divides $f_{I,n}$, 
so $\ell$ must divide 
$-f_{\sigma(I),n} + \ep_I f_{I,m}$
by \cref{sf-f}
giving a contradiction.
Thus $\ell^2$ divides $f_{I,n}$.\\

\noindent \textbf{Case 2: $\mathbf{e>1}$.}
Here,
$s(f_{\sigma(I),m})=\lambda f_{\sigma(I),m}$
by \cref{diagonalizableaction}
so
$\ell$ divides $f_{\sigma(I),m}$ by \cref{modeh}(a). 
Also, by~\cref{transvectionaction},
$$
t_m(f_{\sigma(I),m})=f_{\sigma(I),m}-f_{\sigma(I),n}+\ep_I f_{I,m}-\ep_I f_{I,n},
$$ 
so $\ell^2$ divides $-f_{\sigma(I),n}+\ep_I f_{I,m} - \ep_I f_{I,n}$ by \cref{sf-f}.
Recall again that $\ell$ divides $f_{I,n}$ by part a)  so $\ell$ also divides $-f_{\sigma(I),n}+\ep_If_{I,m}$.
Further, $s$
fixes
$-f_{\sigma(I),n}+\ep_If_{I,m}$
(see \cref{diagonalizableaction})
so $\ell^{2}$ divides $-f_{\sigma(I),n}
+\ep_I f_{I,m}$ by \cref{modeh}(b)
 as $e>1$. Therefore $\ell^2$ divides $f_{I,n}$. 
Finally,
$s(f_{I,n})=\lambda^{-1}f_{I,n}$, 
so $\ell^{\ts e+1}$ divides $f_{I,n}$ by \cref{modeh}(d).\\

\noindent \textbf{Case 3: $\mathbf{b>1}$.}
First, if $I\cap[b]\neq \{m\}$, then $\ell$ divides $f_{I,m}$ by part d).
Then by \cref{transvectionaction},
$f_{I,n}=f_{I,m}-t_m(f_{I,m})$ 
, so $\ell^2$ divides $f_{I,n}$ by \cref{sf-f}. 
Otherwise,
if $I\cap[b]=\{m\}$, take $m'\in [b]$ with $m\neq m'$, so $m'\notin I$. Then 
by \cref{transvectionaction},
$t_m(f_{\sigma(I),m'})
=f_{\sigma(I),m'}+\ep_If_{I,m'}$,
so $\ell$ divides $f_{I,m'}$ 
and
 $\ell^2$ divides
$f_{I,n}=f_{I,m'}-t_{m'}(f_{I,m'})$
by \cref{sf-f}.
\\

\noindent \textbf{Case 4: $\mathbf{\text{char } \FF=2, \, e=1, \, b=1},$ $\mathbf{G}$ contains multiple transvections.}
In this case, $m=1$
and $G$ contains the transvection $t_1$ as well as a transvection $t_1^{(\alpha)}$ with root vector $\alpha v_1$ for some $\alpha\in\FF$ that is not 0 or 1:
$$t_1 = 
\begin{smallpmatrix}
1&&1\vspace{-0.5em}\\
&\ddots\\
&&1
\end{smallpmatrix}
,\quad
t_1^{(\alpha)} = 
\begin{smallpmatrix}
1&&\alpha\vspace{-0.5em} \\
&\ddots\\
&&1
\end{smallpmatrix}.$$
Then since $\characteristic \FF=2$,
for
$\sigma=(1\ n)$,
\begin{small}
$$\begin{aligned}
t_1^{(\alpha)}(x_I)&=
\begin{cases}
x_I&\text{when $n\in I$ \,\, or $1\notin I$},\\
x_I+\alpha x_{\sigma(I)}&\text{when $n\notin I$ and $1\in I$},
\end{cases}
\ \qquad
\text{and }\quad
t_1^{(\alpha)}(v_j)&=
\begin{cases}
v_j&\;\,
\text{when $j\neq n$},
\\
\alpha v_1+v_n&\;\,
\text{when $j=n$}\, .
\end{cases}
\end{aligned}$$ 
\end{small}
Taking
sums over subsets $I\in\tbinom{[n]}{k}$ and  $1\leq j\leq n$,
we observe after some computation that
\begin{small}
$$
\begin{aligned}
t_1^{(\alpha)}(\eta)
=
&
\sum_{\substack{I,j :\\n\notin I\text{ or }1\in I,\\j\neq 1}}
t_1^{(\alpha)}(f_{I,j})\otimes x_I\otimes v_j\ 
+
\sum_{\substack{I :\\n\notin I\text{ or }1\in I}}
\big(t_1^{(\alpha)}(f_{I,1})+\alpha\, t_1^{(\alpha)}(f_{I,n})\big)\otimes x_I\otimes v_1
\\
&\;\;\;\;\;
+
\sum_{\substack{I,j :\\n\in I,1\notin I,\\j\neq 1}}
\big(t_1^{(\alpha)}(f_{I,j})+\alpha\, t_1^{(\alpha)}(f_{\sigma(I),j})\big)\otimes x_I\otimes v_j\\
&\;\;\;\;\;
+\sum_{\substack{I :\\n\in I,1\notin I}}
\big(t_1^{(\alpha)}(f_{I,1})+\alpha\, t_1^{(\alpha)}(f_{I,n})+\alpha\, t_1^{(\alpha)}(f_{\sigma(I),1})+\alpha^2\, t_1^{(\alpha)}(f_{\sigma(I),n})\big)\otimes x_I\otimes v_1
\, .
\end{aligned}
$$
\end{small}%
We equate polynomial coefficients of $\eta$ and $t_1^{(\alpha)}(\eta)$ to deduce that
\begin{small}
\begin{equation}
\label{othertransvectionaction}
t_1^{(\alpha)}(f_{I,j})=
\begin{cases}
f_{I,j}
&\text{for $j\neq 1$ when $n\notin I$\;\;\, or $1\in I$}
,
\\
f_{I,1}+\alpha f_{I,n}
&\text{for $j=1$ when $n\notin I$\;\;\, or $1\in I$}
,
\\
f_{I,j}+\alpha f_{\sigma(I),j}
&\text{for $j\neq 1$ when $n\in I$ and $1\notin I$},
\\
f_{I,1}+\alpha f_{I,n}+\alpha f_{\sigma(I),1}+\alpha^2 f_{\sigma(I),n}
&\text{for $j=1$ when $n\in I$ and $1\notin I$}.
\end{cases}
\end{equation}
\end{small}
Suppose by way of contradiction that $\ell^2$ does not divide $f_{I,n}$. 
Recall again that $\ell$ divides $f_{I,n}$ by part a) so some monomial $M$ of degree $1$ in $x_n$ appears in $f_{I,n}$ with nonzero coefficient $c\in\FF$. 
As $t_1(f_{I,n})=f_{I,n}$ by \cref{transvectionaction},
$$
f_{I,n}\in
S^{G'}=
\FF[x_1^2+x_1x_n, x_i : i\neq 1]
\quad\text{for }
G'=\langle t_1 \rangle
$$
(see \cite{SmithBook}), and thus the degree of $M$ in $x_1$ is even (as $x_n$ divides
$f_{I,n}$).

We analyze the coefficients of $M$ and $Mx_1/x_n$ in $f_{I,1}$, $f_{I,n}$, $f_{\sigma(I),1}$, and $f_{\sigma(I),n}$. 
Note that
$$
f_{I,n}=0\cdot Mx_1/x_n + c\cdot M +\text{other terms}
\, .
$$ 
Next, $t_1(f_{I,1})=f_{I,1} + f_{I,n}$ and $t_1(f_{\sigma(I),n})=f_{\sigma(I),n}+f_{I,n}$ by \cref{transvectionaction}, 
so the coefficients of $Mx_1/x_n$ in $f_{I,1}$ and $f_{\sigma(I),n}$ are both equal to $c$ by \cref{Mxm/xn}.
Fix $c',c''\in \FF$ with
$$\begin{aligned}
f_{I,1}=c\cdot Mx_1/x_n + c' \cdot M + \text{other terms},
\text{  }
f_{\sigma(I),n}= c \cdot Mx_1/x_n + c'' \cdot M + \text{other terms}
\, .
\end{aligned}$$
%
%
Now we examine $f_{\sigma(I),1}$.
On one hand,
by \cref{transvectionaction} and \cref{othertransvectionaction},
\begin{small}
$$
\begin{aligned}
t_1(f_{\sigma(I),1})+f_{\sigma(I),1}
&=f_{\sigma(I),n}+f_{I,1}+f_{I,n}
=0\cdot Mx_1/x_n + (c+c'+c'')\cdot M + \text{other terms}
\text{ and }\\
t_1^{(\alpha)}(f_{\sigma(I),1})+f_{\sigma(I),1}
&=\alpha f_{\sigma(I),n} + \alpha f_{I,1} + \alpha^2 f_{I,n}
=0\cdot Mx_1/x_n + (\alpha^2 c + \alpha c' + \alpha c'') M+\text{other terms}
\, .
\end{aligned}$$
\end{small}%
%
Let $C$ be the coefficient of
$Mx_1/x_n$ in
$f_{\sigma(I),1}$.
Then, on the other hand, since $Mx_1/x_n$ has odd degree in $x_1$,
$$
\begin{aligned}
t_1(f_{\sigma(I),1}) + f_{\sigma(I),1}
=
C \cdot M +\text{other terms},
\quad
t_1^{(\alpha)}(f_{\sigma(I),1}) + f_{\sigma(I),1} 
=
\alpha C\cdot M + \text{other terms},
\end{aligned}
$$
so
$
C=c+c'+c''$
and 
$\alpha C=\alpha^2 c + \alpha c' + \alpha c''$,
which implies that $c=0$ (as $\alpha\neq 0,1$),
giving a contradiction.  
So $\ell^2$ divides $f_{I,n}$.
This completes part f) and the proof of the lemma.
\end{proof}

The next 
lemma
is used to establish
\cref{saitolemma}.
We set $\delta=\delta_H$, which records 
when $G$ comprises only one transvection and the identity (see \cref{deltaH}).

\begin{lemma}
\label{detcoefdiv}
For any set $\B$ of $n\tbinom{n}{k}$ elements in $(S\ot \wedge^k V^*\ot V)^G$, the determinant of $\coef(\B)$ is divisible by
$\ell$ to the power 
$$
\tbinom{n-1}{k}+(e-1)(n-1)\tbinom{n-1}{k-1}
+e
\Big(
(n-\delta)\big(\tbinom{n-1}{k}-\tbinom{n-b-1}{k}\big)
+\tbinom{n-1}{k-1}-\tbinom{n-b-1}{k-1}
\Big).
$$
\end{lemma}
\begin{proof}
The claim follows immediately from \cref{lhdivides}:
\begin{itemize}
\setlength{\itemindent}{-2em}
\item $\ell$ divides each column in a set $A$ of $\tbinom{n-1}{k}$ columns,

\item $\ell^e$ divides each column in a set $B$ of $\tbinom{n-1}{k-1}-\tbinom{n-b-1}{k-1}$ columns,

\item $\ell^{e-1}$ divides each column in a set $C$ of $(n-1)\tbinom{n-1}{k-1}$ columns,

\item $\ell^e$ divides each column in a set $D$ of $(n-1)\tbinom{n-1}{k}-(n-1)\tbinom{n-b-1}{k}-b\tbinom{n-b-1}{k-1}$ columns,

\item $\ell^e$ divides each column in a set $E$ of $b\tbinom{n-b-1}{k-1}$ columns after some column operations, and

\item $\ell^{e(1-\delta)+1}$ divides each column in a set $F$ of  $\tbinom{n-1}{k}-\tbinom{n-b-1}{k}$ columns,
\end{itemize}
where the sets $A$, $B$, $C$, $D$, $E$ are pairwise distinct and $F\subset A$. Hence, $\det \coef(\B)$ is divisible by $\ell$ to the power
$|A|+(e-1)|C|+e\big(|B|+|D|+|E|+(1-\delta)|F|\big).$
\end{proof}

\section*{Acknowledgements}
The authors thank Hiroaki Terao for helpful and lively discussions.

\end{document}